\documentclass[12pt,a4paper]{article}

\usepackage{amssymb}
\usepackage{amsmath, amsthm}
\usepackage{arydshln}
\usepackage{epsfig}
\usepackage{setspace}

\usepackage{geometry}

\def\RR{{\mathbb{R}}}

\def\CC{{\mathbb{C}}}

\numberwithin{equation}{section}

\newcommand{\dom}{\mathop{\rm dom} }
\newcommand{\interior}{\mathop{\rm int} }

\newcommand{\ri}{\mathop{\rm ri} }
\newcommand{\diag}{\mathop{\rm diag} }
\newcommand{\trace}{\mathop{\rm tr} }

\newtheorem{Thm}{Theorem}[section]
\newtheorem{Prop}[Thm]{Proposition}
\newtheorem{Lem}[Thm]{Lemma}

\theoremstyle{definition}

\newtheorem{Rem}[Thm]{Remark}
\newtheorem{Ex}[Thm]{Example}

\title{Convex analysis on Hadamard spaces \\ and scaling problems}
\author{Hiroshi HIRAI \\
Graduate School of Mathematics,   \\
Nagoya University, Nagoya, 464-8602, Japan.\\
\texttt{\normalsize hirai.hiroshi@math.nagoya-u.ac.jp}
}
\begin{document}
\maketitle
\begin{abstract}
		In this paper, we address the bounded/unbounded determination
	of geodesically convex optimization on Hadamard spaces.
	In Euclidean convex optimization, 
	the recession function is a basic tool to study the unboundedness, 
	and provides the domain of the Legendre-Fenchel conjugate of the objective function.
	In a Hadamard space, the asymptotic slope function (Kapovich, Leeb, and Millson 2009), 
	which is a function on the boundary at infinity, 
	plays a role of the recession function.
	We extend this notion by means of convex analysis and optimization, 
	and develop a convex analysis foundation for the unbounded determination 
	of geodesically convex optimization on Hadamard spaces, 
    particularly on symmetric spaces of nonpositive curvature.
	We explain how our developed theory is applied to operator scaling 
	and related optimization on group orbits, which are our motivation.
\end{abstract}

Keywords: Convex analysis, Hadamard space, CAT(0) space, recession function, 
Legendre-Fenchel conjugate, symmetric space, Euclidean building, matrix and operator scaling, null-cone membership, moment polytope, submodular function, Busemann function
\section{Introduction}
{\em Hadamard spaces} are complete geodesic metric spaces having nonpositive curvature.
In such a space, a geodesic connecting any two points is uniquely determined.
A function on a Hadamard space 
is called {\em (geodesically) convex} if it is convex along any geodesics. 
The theory of convex optimization on Hadamard spaces is a promising direction of research, 
though it has just started and is still undeveloped; see e.g., \cite{BacakBook}. 
Since the influential paper~\cite{GGOW} by Garg, Gurvits, Oliveira, and Wigderson
on {\em operator scaling}~\cite{Gurvits04}, 
apparently unrelated problems in diverse fields of mathematical sciences have been formulated and partially/completely solved 
via geodesically convex optimization on Riemannian manifolds; see \cite{AGLOW,BGOWW_tensor0,BFGOWW,Franks18,GGOW_GAFA,HNW23} and references therein.
These manifolds are, in fact, 
{\em Hadamard manifolds} (Riemannian manifolds that are Hadamard spaces), 
more specifically, {\em symmetric spaces of nonpositive curvature}.
%

In these problems, as well as finding near-optimal solutions, 
deciding boundedness of the optimization problem,   
\begin{equation}
\mbox{inf}_{x \in X} \ f(x) > - \infty \mbox{ or } = - \infty
\end{equation}
becomes an important issue. Examples are the approximate scalability in operator scaling
and its invariant theoretic generalizations (null-cone membership, moment polytope membership); see the above references.

The present paper addresses this bounded/unbounded determination 
by means of convex analysis.
For explaining our approach, let us recall the Euclidean situation.
In Euclidean convex optimization, 
 {\em recession functions} (also called {\em asymptotic functions}) are a basic tool to study the boundedness property; see
\cite[Section 3.2]{FundamentalsConvexAnalysis} and \cite[Section 8]{Rockafellar}.
For a convex function $f: \RR^n \to \RR$, 
the recession function $f^{\infty}:\RR^n \to \RR \cup \{\infty\}$ 
is defined by
\begin{equation}\label{eqn:Euclidean_recession}
f^{\infty}(u) := \lim_{t \to \infty}f(x+tu)/t \quad (u \in \RR^n),
\end{equation}
where $f^\infty(u)$ is independent of $x \in \RR^n$.
The recession function $f^{\infty}$ is a positively homogeneous convex function,
and links with {\em Legendre-Fenchel duality} as follows.
Recall the {\em Legendre-Fenchel conjugate} $f^*:\RR^n \to \RR^n \cup \{\infty\}$ of $f$, 
which is defined by $f^*(p) := \sup_{x \in \RR^n} \langle p,x \rangle - f(x)$.
 Its domain $\dom f^* :=  \{p \in \RR^n \mid f^{*}(p) < \infty\}$ 
 is precisely the set of vectors $p \in \RR^n$ for which $\inf_{x \in \RR^n} f(x) - \langle p,x\rangle$
 is bounded below. Then, the recession function $f^{\infty}$ 
 equals the  {\em support function} of the domain $\dom f^*$. 
 Namely,  
 $f^{\infty}(u) = \sup_{p \in \dom f^*} \langle u, p \rangle$ holds,  and the closure $\overline{\dom f^*}$ of $\dom f^*$ equals 
\begin{equation}\label{eqn:Euclidean_B(f^infty)}
B(f^{\infty}) := \{ p \in \RR^{n} \mid  \langle u, p \rangle \leq f^{\infty}(u) \ (u \in \RR^n) \}.
\end{equation}
See \cite[Example 2.4.6]{FundamentalsConvexAnalysis} and \cite[Theorem 13.3]{Rockafellar}. In particular, 
$f^{\infty}$ provides 
an inequality description of $\overline{\dom f^*}$. 

Suppose further that $f$ is smooth. 
The image $\nabla f(\RR^n)$ of gradient map $x \mapsto \nabla f(x)$ 
is precisely the set of vectors $p$ 
for which the infimum of $\inf_{x \in \RR^n} f(x) - \langle p,x\rangle$ is attained.
Then, the gradient space $\nabla f(\RR^n)$ contains 
the relative interior 
$\ri B(f^{\infty}) $ of $B(f^{\infty}) = \overline{\dom f^*}$~\cite[Corollary 26.4.1]{Rockafellar}.
Thus, the following relation holds:
\begin{equation}\label{eqn:relation_Euclidean}
\ri B(f^{\infty}) \subseteq \nabla f(\RR^n) \subseteq \dom f^* \subseteq  B(f^{\infty}).
\end{equation}
Particularly, $\overline{\nabla f(\RR^n)} = \overline{\dom f^*} = B(f^{\infty})$ holds.
If $\dom f^*$ is closed, then the following conditions are equivalent for $p \in \RR^n$: 
\begin{itemize}
	\item[(a)] $\inf_{x \in \RR^n} \|\nabla f(x) -p\|=0$. 
	\item[(b)] $-f^*(p) =\inf_{x \in \RR^n} f(x) - \langle p, x \rangle > -\infty$.
	\item[(c)] $p \in B(f^{\infty})$. 
\end{itemize} 

This equivalence plays fundamental roles 
in 
{\em matrix scaling}~\cite{Sinkhorn64}---the origin of operator scaling 
and related group orbits optimization.
Given an $n \times n$ nonnegative matrix $A =(a_{ij})$ and positive vectors $r,c \in \RR^n$ 
with the same sum $\sum r_i = \sum_i c_i = l$, 
the matrix scaling problem is to ask positive diagonal matrices $R,C$ 
such that $((RAC)^{\top}{\bf 1}, RAC{\bf 1}) = (r,c)$, where ${\bf 1}$ denotes the all-ones vector. 
The matrix $A$ is said to be {\em scalable} if there exist such $R,C$, 
and {\em approximately scalable} if there exist $R,C$ 
such that $((RAC)^{\top}{\bf 1}, RAC{\bf 1})$ is arbitrary close to $(r,c)$.
In fact, the approximate scalability is written as the condition (a) for convex function
$f_A(s,t) := l \log \sum_{i,j} e^{s_i} a_{ij} e^{t_j}$ and vector $p=(r,c)$.
The recession function $f^{\infty}_A$ is given by
$f^{\infty}_A(u,v) = l\max\{ u_i+v_j \mid i,j: a_{ij} \neq 0\}$.
The condition (c) is written as a network flow LP and efficiently verified. 
The combinatorial characterization of the approximate scalability by Rothblum and Schneider~\cite{RothblumSchneider} was proved via
$f^{\infty}$ in this way. 
From (b) and the fact that $\dom f^*_A$ is closed, approximate scaling $RAC$ 
is obtained by solving convex optimization $\inf_{s,t} f_A(s,t)- \langle r,s \rangle - \langle c,t\rangle$.
{\em Sinkhorn algorithm}~\cite{Sinkhorn64} is viewed 
as an alternating minimization algorithm 
for this problem.

The goal of this paper is to establish analogues of
the equivalence of (a), (b), and (c) 
for geodesically convex optimization on Hadamard spaces, and to provide a convex analysis foundation to the above mentioned problems.
We particularly focus on an analogous notion of the recession function 
on a Hadamard space. 
In fact, such a notion was already introduced by 
Kapovich, Leeb, and Millson~\cite{KLM2009JDG} (see also \cite[Section 5.4]{Woodward}), who defined 
the {\em asymptotic slope} of a convex function via a geodesic analogue 
of (\ref{eqn:Euclidean_recession}).
By utilizing the asymptotic slope,
they studied the boundedness (semistability) 
for a class of convex functions related to a generalization of Horn's problem.
Our presented theory reformulates and extends some of their arguments 
from viewpoints of convex analysis and optimization.

The structure and results of this paper are outlined as follows. 
In Section~\ref{sec:Hadamard}, 
we provide necessary backgrounds on a Hadamard space $X$,  particularly, the {\em boundary $X^{\infty}$ at infinity} and its {\em Euclidean cone} $CX^{\infty}$. 
 The cone $CX^{\infty}$ of the boundary is also a Hadamard space, and 
 plays a role of the ``dual space" of $X$.
 With a point $p \in CX^{\infty}$, we associate the {\em Busemann function} $b_p$ 
 as a correspondent of a linear function, 
 and regard the boundary $CX^{\infty}$ as the space of Busemann functions.
 This viewpoint leads to the notion of the {\em asymptotic Legendre-Fenchel conjugate} $f^*$ of $f$, which is a function on $CX^{\infty}$ defined by $p \mapsto f^*(p) := \sup_{x \in X} -b_p(x) - f(x)$.
 The asymptotic slope is a function on $X^{\infty}$. 
 We extend it, in a homogeneous way, 
 to introduce the recession function $f^{\infty}$ on $CX^{\infty}$. 
  We also define the associated subset $B(f^{\infty})$ 
 via an analogue of (\ref{eqn:Euclidean_B(f^infty)}).
 We show that $f^{\infty}$ 
 is positively homogeneous convex on $CX^{\infty}$  
 and that $p \in B(f^{\infty})$ is a necessary condition for $f^{*}(p) < \infty$, 
 i.e., $\dom f^* \subseteq B(f^{\infty})$.
 
 We then move to smooth convex optimization on a Hadamard manifold $M$. 
 The boundary cone $CM^{\infty}$ is identified with the tangent space $T_x$ at any point $x$, 
 though $CM^{\infty}$ has a different topology from the usual one $T_x \simeq \RR^n$, 
 and is not a manifold in general.
 We define the {\em asymptotic gradient} 
 $\nabla^{\infty}f(x)$ as the image of the gradient $\nabla f(x)$ in $CM^{\infty}$.
 This notion fits into our purpose: 
 We verify a weaker analogue of (\ref{eqn:relation_Euclidean}) 
 in which $\nabla$ is replaced by $\nabla^{\infty}$ 
 and $\ri$ is placed by the interior.
 We also verify an analogue of the equivalence of (a) and (c).
  
 We further restrict our study to 
 {\em symmetric spaces of nonpositive curvature}---a representative class 
 of Hadamard manifolds having rich group symmetry.
 The boundary $CM^{\infty}$ has a polyhedral cone complex structure, 
 known as a {\em Euclidean building}. 
 We utilize the symmetry property and building structure to 
 show that the conjugate $f^*$ is convex on each cone ({\em Weyl chamber}).
 We also provide some fundamental results on $\dom f^*$ and $B(f^{\infty})$.
 Then we present detailed calculations and specializations of several concepts 
 for the symmetric space  $P_n = GL(n,\CC)/U(n)$ of positive definite matrices.
 The boundary $CP_n^{\infty}$ is viewed as 
 the order complex of flags of vector subspaces, which 
 naturally links {\em submodular functions} 
 on the lattice of vector subspaces and their {\em Lov\'asz extension}~\cite{HamadaHirai,HH16L-convex}. Via the Lov\'asz extension,  
 submodular functions correspond to convex functions on $CP_n^{\infty}$ that are linear 
 on each cone.
 For a convex function $f$ on $M$ inducing a submodular function at infinity, called an {\em asymptotically submodular} function,  
 $B(f^{\infty})$ is an analogue of the {\em base polyhedron}
  and the membership of $B(f^{\infty})$ becomes 
  discrete convex optimization (i.e., {\em submodular function minimization}) 
  over the modular lattice of all vector subspaces of $\CC^n$. 
  
 In Section~\ref{sec:scaling}, we explain how the developed theory is applied 
 to operator scaling and related optimization on group orbits.
 These problems are viewed as convex optimization on symmetric spaces of nonpositive curvature. 
 We take up the {\em operator scaling with marginals} by Franks~\cite{Franks18}.
 We compute the recession function, from which Franks' characterization of 
 the approximate triangular scalability is naturally deduced.

 We finally consider the null-cone and moment polytope membership for a linear action of a reductive group $G$, formulated by 
 B\"urgisser, Franks, Garg, Oliveira, Walter, and Wigderson~\cite{BFGOWW}.
 These are convex optimization of the {\em Kempf-Ness function} $f_v$
 over symmetric space $M = G/K$.
 In this setting, asymptotic gradients give rise to 
 the {\em moment map} and {\em moment polytope}, 
 for which $f_v^{\infty}$ gives an inequality description.
 We establish the link with Hilbert-Mumford criterion and Kempf-Ness theorem, and 
 verify from our approach shifting trick and convexity theorem 
 for the moment polytope.
 
 The main implication of this paper is that
 finding a one-parameter subgroup in the Hilbert-Mumford criterion  
 and membership of the moment polytope (after randomization) 
 reduce to minimization of recession functions $f_v^{\infty}$ 
 over Euclidean building $CM^{\infty}$. 
 This is a far-reaching generalization of the approach
 by Hamada and Hirai~\cite{HamadaHirai}
 for the {\em noncommutative-rank} computation 
 ($=$ the null-cone problem of the left-right action), 
 in which their algorithm is now viewed as minimizing $f_v^{\infty}$.
 The current approach (e.g.,\cite{BFGOWW,HNW23})
 for these problems is mainly based on  
 smooth convex optimization (b) on $G/K$. 
 We hope that our results will motivate to 
 develop nonsmooth convex optimization techniques 
 on non-manifold Hadamard spaces, for verifying~(c).


\section{Convex analysis on Hadamard spaces}\label{sec:Hadamard}

\subsection{Hadamard spaces}\label{subsec:Hadamard}
Here we introduce Hadamard spaces; see \cite{BacakBook,BallmanBook,BrHa} for details.
Let $X$ be a metric space with distance function $d$.
A {\em path} in $X$ is a continuous map from interval 
$[0,l] \subseteq \RR$ to $X$, where $l \geq 0$.  
We say that a path $c:[0,l] \to X$ connects $c(0)$ and $c(l)$.
A path $c:[0,l] \to X$ is said to be {\em geodesic} 
if $d(c(s), c(t)) = |s-t|$ for every $s,t \in [0,l]$.
A {\em geodesic metric space} is a metric space $X$ in which
any pair of two points is connected by a geodesic path.

Suppose that $X$ is a geodesic metric space.
Let $x,y,z \in X$. A {\em geodesic triangle} 
of $x,y,z$ is the union of geodesic paths connecting $x,y$, $x,z$, and $y,z$.
The {\em comparison triangle} is the triangle in Euclidean plane $\RR^2$ with vertices $\overline{x}, \overline{y}, \overline{z}$ such that  $d(x,y) = \|\overline{x} - \overline{y}\|_2$, 
$d(y,z) = \|\overline{y} - \overline{z}\|_2$ and $d(z,x) = \|\overline{z} - \overline{x}\|_2$.
In the geodesic triangle, 
suppose that points $x,y$ and $x,z$ are connected by geodesic paths $c$ 
and $c'$, respectively, 
with $c(0) = c'(0) = x$. 
For $t,t' \in [0,1]$, 
let $p := c(t d(x,y))$ and $q := c'( t' d(x,z))$. 
Let $\overline{p} := (1-t) \overline{x} + t \overline{y}$ and 
$\overline{q} := (1-t') \overline{x} + t' \overline{z}$ be the corresponding points in $\RR^2$.
Then the {\em CAT(0)-inequality} is given by
\begin{equation}\label{eqn:CAT(0)}
d(p,q) \leq \| \overline{p} - \overline{q} \|_2.
\end{equation}
A geodesic metric space $X$ is said to be {\em CAT(0)} 
if the CAT(0) inequality (\ref{eqn:CAT(0)})
holds for every choice of a geodesic triangle and $t \in [0,1]$.
There are several ways of defining CAT(0) spaces.
A useful one is the following: $X$ is CAT(0) if
for every triple 
of points $x,y,z \in X$, geodesic path $c$ with $c(0) = x$ and $c(d(x,y)) = y$, 
and $t \in [0, 1]$, it holds
\begin{equation}\label{eqn:CAT(0)_2}
d(p_t, z)^2 \leq (1-t) d(x,z)^2 + t d(y,z)^2 - t(1-t) d(x,y)^2, 
\end{equation}
where $p_t = c(t d(x,y))$. 
Notice that 
the RHS equals the squared comparison distance 
between $\overline{p_t} = (1-t) \overline{x} + t \overline{y}$ and $\overline{z}$.

It is known \cite[II.1.4]{BrHa} that CAT(0) spaces are {\em uniquely geodesic}, i.e., a geodesic path connecting any two points is unique.
In this case, for points $x,y \in X$, let $[x,y]$ denote the image of the unique geodesic path connecting $x,y$.
For $t \in [0,1]$, 
let $(1-t) x + t y$ denote the point $z \in [x,y]$ 
with $d(x,z)/d(x,y) = t$.

A {\em Hadamard space} is a CAT(0) space that is complete as a metric space. 
Let $X$ be a Hadamard space. A subset $S \subseteq X$ is called {\em convex} 
if $x,y \in X$ implies $[x,y] \subseteq X$.
The smallest convex set including a given $S \subseteq X$ is called the {\em convex hull} of $S$.
A function $f: X \to \RR \cup \{\infty\}$ is said to be {\em convex}
if for all $x,y \in X, t \in [0,1]$
it satisfies 
\begin{equation}\label{eqn:convexity}
(1-t) f(x) + t f(y) \geq f((1-t) x + t y).
\end{equation}
$f$ is called {\em strictly convex} if $<$ holds
in (\ref{eqn:convexity}) for any $t \in (0,1)$. 
If $-f$ is convex, then $f$ is said to be {\em concave}.
An {\em affine function} is a convex and concave function.
A function $f:X \to \RR$ is said to be {\em $L$-Lipschitz}
with parameter $L \geq 0$
if 
it satisfies $|f(x) - f(y)| \leq L d(x,y)$
for all $x,y \in X$.

\subsubsection{Boundary at infinity and its Euclidean cone}\label{subsub:boundary}
To study the asymptotic behavior of convex functions, 
we consider the boundary of a Hadamard space $X$; 
see  \cite[Chapter II]{BallmanBook} and \cite[Chapter II.8]{BrHa}
for the boundary.

A {\em geodesic ray} is a continuous map $c:[0, \infty) \to X$ 
such that $d(c(s),c(t)) = |s-t|$ for $s,t \in [0,\infty)$.
With a geodesic ray $c$ and $r \geq 0$, the map written as $[0,\infty) \ni t \mapsto c(rt)$
is called a {\em constant-speed ray}, or simply, a {\em ray}.
If $c(0) = x$, we say that ray $c$ {\em issues from} $x$.
Two geodesic rays $c,c'$ are said to be {\em asymptotic}
if there is a positive constant $K > 0$ such that $d(c(t),c'(t)) \leq K$ for all $t \geq 0$.
The asymptotic relation is an equivalence relation on the set of all geodesic rays. 
Let $X^{\infty}$ denote the set of all equivalence classes, 
which is called the {\em boundary of $X$ at infinity}.
The equivalence class of $c$ is called the {\em asymptotic class} of $c$, 
and is denoted by $c(\infty)$.
For any point $x \in X$, and each $\xi \in X^{\infty}$ 
there is a unique geodesic ray $c$ issuing from $x$ such that $c(\infty) = \xi$.
Therefore, $X^{\infty}$ can be identified with the set of 
all geodesic rays issuing from any fixed $x$.

We next introduce a metric on $X^{\infty}$.
For points $x,y,z \in X$ (with $y \neq x \neq z$), 
the {\em comparison angle} $\overline{\angle}_{x}(y,z)$ between $y$ and $z$ at $x$
is the inner angle of the comparison triangle of $x,y,z$ at $x$.
For two geodesic rays $c,c'$ issuing from $x$, let the angle $\angle_x (c,c')$ of $c,c'$ at $x$ 
be defined by
\begin{equation}
\angle_x(c,c') := \lim_{t \to 0} \overline{\angle}_{x}(c(t),c'(t)),
\end{equation}
where the limit indeed exists~\cite[II.3.1]{BrHa}. 
For  $\xi, \xi' \in X^{\infty}$,  
the {\em angle} $\angle (\xi, \xi') \in [0,\pi]$ is defined by
\begin{equation}\label{eqn:angle}
\angle (\xi, \xi') := \sup_{x \in X}  \angle_x (c,c') =  \lim_{t,t' \to \infty} \overline{\angle}_x(c(t),c'(t')) = \sup_{t,t' > 0} \overline{\angle}_x(c(t),c'(t')),   
\end{equation}
where $c,c'$ are geodesic rays issuing from $x$ with
$c(\infty) = \xi$, $c'(\infty) = \xi'$. The equalities in (\ref{eqn:angle}) follow
from \cite[{II. 9.8 (1)}]{BrHa} (where $x$ is arbitrary).
Here $(\xi,\xi') \mapsto \angle (\xi, \xi')$ defines a metric on $X^{\infty}$, 
which is called the {\em angular metric}.
Then $X^{\infty}$ becomes a metric space, and a topological space accordingly.

In addition to $X^{\infty}$, 
we consider its {\em Euclidean cone} $CX^{\infty}$; 
see \cite[Chapter I.5]{BrHa} for generalities of the Euclidean cone construction.
Let $\RR_+$ denote the set of nonnegative numbers.
The set $CX^{\infty}$ is the quotient of $\RR_+ \times X^{\infty}$ 
by the equivalence relation 
$(r, \xi) \simeq (r', \xi')$ if $r =r'=0$ or $(r,\xi) = (r',\xi')$.
The equivalence class of $(r,\xi)$ is denoted by $r\xi$.
By identifying $\xi$ with $1\xi$, 
we regard $X^{\infty}$ as a subset of $CX^{\infty}$.
Let $0$ denote the class of $(0,\xi)$.
The angle $\angle (p,p')$ of (nonzero) points $p= r\xi, p' =r'\xi' \in CX^{\infty}$ 
is defined as $\angle (\xi,\xi')$.
%

The space $CX^{\infty}$ is viewed as the space 
of all constant-speed rays issuing from any fixed point.
Indeed, $r \xi \in CX^{\infty}$ is associated with 
a constant-speed ray $t \mapsto c^*(t) = c(r t)$, 
where $c$ is a geodesic ray with $c(\infty) = \xi$.
In this case, we let $c^*(\infty) := r\xi$.
The space $CX^{\infty}$ is metrized by the following distance $d^{\infty}$:
\begin{equation}\label{eqn:d^infty}
d^{\infty}(r\xi, r'\xi')^2 := r^2 + r'^2 - 2 r r' \cos  \angle (\xi, \xi')  \quad (r\xi, r'\xi' \in CX^{\infty}).
\end{equation}
The topology of $CX^{\infty}$ is given accordingly, which is called the {\em $d^{\infty}$-topology}. 

\begin{Lem}[{see \cite[II.4.8]{BallmanBook}}]\label{lem:CX^infty}
$CX^{\infty}$ is a Hadamard space.
\end{Lem}

A {\em flat triangle} in a Hadamard space is 
a geodesic triangle whose convex hull is isometric to 
the convex hull of their comparison triangle in $\RR^2$.
It is known (see \cite[II.2.9]{BrHa})
that if one of the vertex angles of the triangle is equal to 
the corresponding angle of the comparison triangle, then it is flat.
By construction of $CX^{\infty}$, the angle at $0$ is always equal 
to the corresponding angle of the comparison triangle. Hence we have:
 \begin{Lem}\label{lem:flat}
	In $CX^{\infty}$, 
	any three points containing $0$ form a flat triangle.
\end{Lem}
Therefore, the convex hull of 
$\{ 0, tp,t'p'\}_{t,t' \in \RR_+}$ is isometric to 
the convex cone $C$ in $\RR^2$.
Then, via the isometry to $C \subseteq \RR^2$, 
we can consider nonnegative combinations 
in two nonzero points in $CX^{\infty}$.
%
For a point $p = t\xi \in CX^{\infty}$ and $a \in \RR_+$, 
define $ap := (at) \xi \in CX^{\infty}$.
It is clear that $p \mapsto ap$ is continuous and $a(bp) = (ab) p$.
For two points $p,q \in CX^{\infty}$, 
the sum $p+q$ of $p,q$ is defined by
\begin{equation*}
p + q := 2 ((1/2) p + (1/2) q).
\end{equation*}
Recall that $(1/2) p + (1/2) q$ is the midpoint of the geodesic path between $p$ and $q$.
Map $(p,q) \mapsto p+q$ is continuous; 
this follows from the fact~\cite[II.1.4]{BrHa} that 
the geodesic segment in a Hadamard space varies continuously with its endpoints. 
Notice that $+$ is not associative in general; 
so there may be many ``inverses" $q$ of $p$ with $p+q = 0$.
From (\ref{eqn:d^infty}), it holds $d^{\infty}(\alpha p, \alpha q) = \alpha d^{\infty} (p,q)$.
This means that by $p \mapsto \alpha p$, 
geodesic segment $[p,q]$ is mapped to geodesic segment $[\alpha p, \alpha q]$. 
Then we have a linearity relation
\begin{equation}
\alpha (p+q) = \alpha p + \alpha q  \quad (\alpha \in \RR_+).
\end{equation}
Define the inner product $\langle p, q \rangle$ of 
$p,q \in CX^{\infty}$
by
\begin{equation}\label{eqn:<p,q>}
2 \langle p, q \rangle := d^{\infty}(0,p)^2 + d^{\infty}(0,q)^2 - d^{\infty}(p,q)^2. 
\end{equation}
For  $p = t \xi \in CX^{\infty}$, 
define the norm $\|p \|$ by $\|p \| := t$.
Observe from definitions (\ref{eqn:d^infty}) (\ref{eqn:<p,q>}) 
that $\|p\| = d^{\infty}(0,p) = \sqrt{\langle p,p \rangle}$. Then we have
\begin{equation}\label{eqn:<p,q>_2}
\langle p,q \rangle =\|p\| \|q\| \cos \angle (p, q).
\end{equation}

A function $f:CX^{\infty} \to \RR \cup \{\infty\}$ is called {\em positively homogeneous} 
if it holds
\begin{equation*}
f(\alpha p) = \alpha f(p) \quad (a \in \RR_+).
\end{equation*}

\begin{Lem}\label{lem:<p,q>}
For any $q \in CX^{\infty}$,
the function $p \mapsto \langle p, q \rangle$ is continuous 
and positively homogeneous concave.
\end{Lem}
\begin{proof}
	The continuity follows from (\ref{eqn:<p,q>}) that
	$\langle, \rangle$ is written by continuous function $d^{\infty}$. 
    The positive homogeneity follows from (\ref{eqn:<p,q>_2}).
    For concavity, it suffices to show
\begin{equation}\label{eqn:<p+p',q>}
\langle p+p', q \rangle \geq \langle p, q \rangle +  \langle p', q \rangle. 
\end{equation}
Twice the RHS is equal to
\begin{equation}\label{eqn:RHS}
- d^{\infty}(p,q)^2 - d^{\infty} (p',q)^2 +\|p\|^2 + \|p'\|^2 + 2 \|q\|^2. 
\end{equation}
On the other hand, twice the LHS is equal to
\begin{equation}\label{eqn:LHS}
4 \langle (p+p')/2,q \rangle   = - 2 d^{\infty}((p+p')/2,q)^2 + 2 \|(p+p')/2\|^2 + 2\|q\|^2.
\end{equation}
Since $0,p,p'$ form a flat triangle (Lemma~\ref{lem:flat}), the CAT(0) inequality (\ref{eqn:CAT(0)_2}) with $(x,y,z) = (p,p',0)$ and $t=1/2$ holds in equality: 
\begin{equation}\label{eqn:LHS'}
2 \|(p+p')/2\|^2 = \|p\|^2 + \|p'\|^2 - d^{\infty}(p,p')^2/2. 
\end{equation}
From (\ref{eqn:RHS}), (\ref{eqn:LHS}), and (\ref{eqn:LHS'}), we see
that (\ref{eqn:<p+p',q>}) is equivalent to
\begin{equation}\label{eqn:<p+p',q>=>d} 
2 d^{\infty}((p+p')/2,q)^2 \leq d^{\infty}(p,q)^2 + d^{\infty} (p',q)^2 - d^{\infty}(p,q)^2/2.
\end{equation}
This is the CAT(0)-inequality (\ref{eqn:CAT(0)_2}) with $(x,y,z) =(p,p',q)$ and $t=1/2$, which holds by Lemma~\ref{lem:CX^infty}.
\end{proof}


\begin{Ex}\label{ex:Euclidean_X^infty}
	Consider the case of $X = \RR^n$.
	Two geodesic rays $t \to \xi t + b$ and $t \to \xi' t + b'$ are asymptotic 
	if and only if $\xi = \xi'$. 
	The angle of the corresponding asymptotic classes 
	is given by $\cos^{-1} \langle \xi, \xi' \rangle$.
	Therefore, $X^{\infty}$ is isometric to the sphere $S^{n-1}$.
	The Euclidean cone $CX^{\infty}$ is isometric to 
	the Euclidean space $\RR^n = \RR_{+} \times S^{n-1} / \simeq$.
	Also the distance $d^{\infty}$ and product $\langle,\rangle$ coincides with the Euclidean ones. 
\end{Ex}

\subsubsection{Busemann functions and asymptotic Legendre-Fenchel conjugate}
%

For a geodesic ray $c$,  
the {\em Busemann function} $b_c:X \to \RR$
is defined by
\begin{equation}
b_c(x) := \lim_{t \to \infty} d(x, c(t)) - t \quad (x \in X).
\end{equation}
\begin{Lem}[{see \cite[II.8.22]{BrHa}}]\label{lem:Busemann_convex}
	Busemann functions for geodesic rays are $1$-Lipschitz convex functions.
\end{Lem}
Two asymptotic geodesic rays $c,c'$ yields the same Busemann functions 
up to additive constant:
\begin{Lem}[{see \cite[II.8.20]{BrHa}}]\label{lem:constant}
$c$ and $c'$ are asymptotic if and only if $b_c - b_{c'}$ is a constant function.
\end{Lem}
The Busemann function $b_{c}$ is extended for 
a constant-speed ray $c: t \mapsto c^*(rt)$ by $b_c := r b_{c^*}$, where $c^*$ is a geodesic ray.
Fix $x_0 \in X$.  
For $p \in CX^{\infty}$, there is a unique ray $c$ issuing $x_0$ with $c(\infty) = p$, 
and hence $b_c$ is also written as $b_{x_0,p}= b_{p}$.
In this way, 
we can identify $CX^{\infty}$ 
with the space of Busemann functions.
\begin{Ex}\label{ex:Euclidean_Busemann}
		In the case of $X = \RR^n$, 
		the Busemann function $b_c$ for a geodesic ray $c(t) = \xi t + x_0$ is given by
	\begin{equation*}
	b_c(x) = -  \langle \xi, x-x_0\rangle.
	\end{equation*}
	This follows from $\| x - \xi t - x_0\|_2 - t = t - \langle \xi, x-x_0\rangle  +  O(1/t) - t$. 
	Thus, Busemann functions (for constant-speed rays) are precisely affine functions.
	If the origin $0$ is chosen as $x_0$, 
	the space $CX^{\infty} (\simeq \RR^n)$ is identified with the space of linear functions 
	$x \mapsto - \langle p,x \rangle$, i.e., the dual space of $\RR^n$. 
\end{Ex}
%
For a (convex) function $f: X \to \RR$, define 
the {\em asymptotic Legendre-Fenchel conjugate} 
$f^*:CX^{\infty} \to \RR \cup \{\infty\}$ of $f$ by 
\begin{equation}\label{eqn:f*}
f^{*}(p) := \sup_{x \in X} - b_{p} (x) - f(x) \quad (p \in CX^{\infty}).
\end{equation}
Note that this notion of Legendre-Fenchel conjugate is rather different from those introduced by \cite{BHLTV2021,LBH2022}.
If $X = \RR^n$, then this matches 
the usual Legendre-Fenchel conjugate (see Examples~\ref{ex:Euclidean_X^infty} and \ref{ex:Euclidean_Busemann}).
Notice that $f^*$ is not necessarily convex in $CX^{\infty}$.
We give a simple lemma providing explicit examples of asymptotic Legendre-Fenchel conjugates.
\begin{Lem}
	For a function $h: \RR_+ \to \RR$, define $f:X \to \RR$ by
	\[
	f(x) := h(d(x,x_0)) \quad (x \in X).
	\] 
	Then  the asymptotic Legendre-Fenchel conjugate $f^*$ is given by
	\[
	f^*(p) := h^*(\|p\|) \quad (p \in CX^{\infty}),
	\]
	where $h^*$ is the Legendre-Fenchel conjugate of $h$, i.e.,
	$h^*(\xi) := \sup_{r \in \RR} \xi r- h(r)$ with $h(r) := \infty$ for $r < 0$.
\end{Lem}
\begin{proof}
	By definitions of $f^*$ and $b_p$, we have
	\begin{equation*}
	f^*(p) = \sup_{x \in X} - \|p\| \lim_{t \to \infty} (d(x,c(t))-t)  - h(d(x,x_0)),
	\end{equation*}
where $c$ is a geodesic ray with $c(0) = x_0$ and $c(\infty) = p/\|p\|$.
For $x \in X$, letting $r := d(x,x_0)$, 
we have $h(d(x,x_0)) = h(r)$, 
and $d(x,c(t)) \geq d(c(t),x_0) - d(x,x_0) = t-r = d(c(r),c(t))$ for every large $t$. 
This implies that $\sup$ can be restricted to $c$.
Thus $f^*(p) = \sup_{r \in \RR_+}  \|p\| r - h(r) = h^*(\|p\|)$.
\end{proof}

\begin{Ex}
	Suppose that $f(x) = \frac{1}{2}d(x,x_0)^2$. By the CAT(0)-inequality (\ref{eqn:CAT(0)_2}), 
	$f$ is (strongly) convex.  By the above lemma, we have
	$f^*(p) = \frac{1}{2}\|p\|^2$. 
	In this case, $f^*$ is also convex. 
	The convexity is seen by considering the flat triangle of vertices $0,p,q$. 
\end{Ex}
Further investigation of $f^*$ is left for future research.
As mentioned in the introduction, our central interest is how to describe the domain
$\dom f^* = \{ p \in CX^{\infty} \mid f^*(p) < \infty\}$ 
of the conjugate $f^*$.

\subsubsection{Recession functions (asymptotic slope functions)}
Let $f: X \to \RR$ be a continuous convex function.
The {\em recession function}
 $f^{\infty}: CX^{\infty} \to \RR \cup \{\infty\}$ of $f$ is defined by
 \begin{equation}\label{eqn:f^infty}
 f^{\infty}(p) := \lim_{t \to \infty} \{f(c(t)) - f(c(0))\}/t = \lim_{t \to \infty} f(c(t))/t \quad (p \in CX^{\infty})
 \end{equation}
 where $c$ is a constant-speed ray with $c(\infty) = p$. 
 The recession function is just a homogeneous extension of the {\em asymptotic slope function} by Kapovich, Leeb, and Millson~\cite{KLM2009JDG}, which is defined on $X^{\infty}$.
 By convexity, $(f(c(t)) - f(c(0)))/t$ is monotone nondecreasing, 
 and it converges to a finite value or $\infty$.
  The recession function is indeed independent of the choice of a ray $c$.
 \begin{Lem}[{\cite[Lemma 2.10]{KL2006}}]
 	$\lim_{t \to \infty} f(c(t))/t = 
 	\lim_{t \to \infty} f(c'(t))/t$ holds for two asymptotic geodesic rays $c,c'$.
 \end{Lem}
 It is easy to see the positive homogeneity of $f^{\infty}$
 (by the change of variable in (\ref{eqn:f^infty})).
 In particular, $f^{\infty}(r\xi) = r f^{\infty}(\xi)$ for $\xi \in X^{\infty}$.
A partial convexity property of asymptotic slope functions
is obtained in \cite[Lemma 3.2 (ii)]{KLM2009JDG}.
In the setting of recession functions 
the following general convexity holds.
\begin{Thm}\label{thm:f^infty}
The recession function  $f^{\infty}$ is positively homogeneous convex.
\end{Thm}
\begin{proof}
We have already seen the positive homogeneity.
Hence it suffices to show
\begin{equation*}
f^{\infty}(p) + f^{\infty}(p') \geq f^{\infty}(p+p') \quad (p,p' \in CX^{\infty}).
\end{equation*}
Let $p,p' \in CX^{\infty}$. We can assume that both $p$ and $p'$ are nonzero 
and both $f^{\infty}(p)$ and $f^{\infty}(p')$ are finite.
Suppose that $p = a \xi$ and $p' = a' \xi'$ 
for $a \geq a' > 0$ and $\xi, \xi' \in X^{\infty}$.
Let $x \in X$ and let $\sigma, \sigma'$ 
be geodesic rays issuing from $x$ with $\sigma(\infty) = \xi$ and $\sigma'(\infty) = \xi'$.

Suppose first that $p + q \neq 0$. 
Then $\angle (\xi,\xi') < \pi$, or 
$\angle (\xi,\xi') = \pi$ and $a > a'$.
Let $m(t)$ be the midpoint of $\sigma(a t)$ and $\sigma'(a't)$. 
We will consider geodesic segment $[x, m(t)]$ with $t \to \infty$.
We show $e := \lim_{t \to \infty} d(x,m(t))/ t > 0$.
Indeed, by triangle inequality $d(x,m(t)) \geq d(x, \sigma(at)) - d(\sigma(at),m(t))$
with $d(x,\sigma(at)) = at$ and $d(\sigma(at), m(t)) = d(\sigma(a t), \sigma'(a' t))/2$ we have
\begin{equation*}
d(x,m(t))/ t  \geq  a - d(\sigma(a t), \sigma'(a' t))/2t =  a  - \frac{1}{2}\sqrt{a^2 + a'^2 -2 a a' 
	\cos \overline{\angle}_x (\sigma(a t), \sigma'(a' t))},
\end{equation*}
where the equality follows from 
the law of cosine in the comparison triangle of $x, \sigma(at), \sigma(a't')$.
Letting $t \to \infty$, we have 
$\overline{\angle}_x (\sigma(a t), \sigma'(a' t)) \to \angle(\xi, \xi')$ and
\begin{equation*}
e \geq a - \frac{1}{2}\sqrt{a^2 + a'^2 -2 a a' 
	\cos \angle (\xi,\xi')} \geq (a-a')/2.
\end{equation*}
By $a > a'$ or $\angle (\xi,\xi') < \pi$,   
we have $e > 0$.
By~\cite[II.4.4]{BallmanBook},
it holds $p + p' = 2 e \eta$ for some $\eta \in X^{\infty}$.
Let $\rho$ denote the geodesic ray with $\rho(0) = x$ 
and $\rho(\infty) = \eta$. 

We can assume $f(x) = 0$ by replacing $f$ with $f- f(x)$.
Consider 
the constant-speed path $\rho_t: [0,et] \to X$ with $\rho(0) = x$ and 
$\rho(et) = m(t)$.
Then, by \cite[II.4.4]{BallmanBook},  
the path $\rho_t$ converges to geodesic ray $\rho$ for $t \to \infty$. 
For every $\epsilon > 0$ and $s \geq 0$, 
there is $t_0$ such that for every $t \geq t_0$ it holds
\begin{equation*}
f(\rho(s)) \leq \epsilon + f(\rho_t(s)) \leq \epsilon + f(m(t))d(x,\rho_t(s))/d(x,m(t)),
\end{equation*}
where the second inequality follows from $f(x) = 0$ and the convexity of $f$ along $[x, m(t)]$. 
For all large $s,t$ we have
\begin{equation*}
\frac{f(\sigma(at))}{t} + \frac{f(\sigma'(a't))}{t} \geq 2\frac{f(m(t))}{t} \geq 
2 \left(\frac{f(\rho(s))}{s} -  \frac{\epsilon}{s}\right) \frac{s}{d(x, \rho_t(s))} \frac{d(x,m(t))}{t}
\end{equation*}
By $t \to \infty$, we have $d(x,\rho_t(s)) \to s$ and $d(x,m(t))/t \to e$.  By $s \to \infty$, we have
\[
f^{\infty}(p) + f^{\infty}(p') \geq \lim_{s \to \infty} f(\rho(s))2e /s = 2e f^{\infty}(\eta)= f^{\infty}(p+p').
\]

Finally, consider the case $p+p' = 0$, 
i.e., $\angle(\xi,\xi') = \pi$ and $a=a'$.
For $\epsilon > 0$, it holds $p+ (1+\epsilon)p' \neq 0$.
Then, from the above case, we have
\[
f^{\infty}(p) + (1+ \epsilon) f^{\infty}(p') = f^{\infty}(p) + f^{\infty}((1+\epsilon)p') \geq 
f^{\infty}(p + (1+\epsilon)p')= \epsilon f^{\infty}(p').
\]
For the last equality, observe from (\ref{eqn:d^infty}) that 
the midpoint between $p,(1+\epsilon)p'$ is $(\epsilon/2) p'$. Hence
$p +(1+\epsilon) p' = \epsilon p'$ and 
$f^{\infty}(\epsilon p') = \epsilon f^{\infty}(p')$.
By $\epsilon \to 0$, we obtain the desired inequality.
\end{proof}
It is known \cite[p. 318]{KLM2009JDG} that 
the asymptotic slope of a Busemann function is given by 
$b_{\xi}(\eta)^{\infty} = - \cos \angle (\xi, \eta)$ for $\xi, \eta \in X^{\infty}$.
Hence we have: 
\begin{Lem}\label{lem:b^infty} 
	It holds $(b_{p})^{\infty}(q) = - \langle p, q \rangle$ for $p,q \in CX^{\infty}$.
\end{Lem}
Motivated by (\ref{eqn:Euclidean_B(f^infty)}), 
for any positively homogeneous function $h:CX^{\infty} \to \RR \cup \{\infty\}$
we define
a subset $B(h) \subseteq CX^{\infty}$ by
\begin{equation}
B(h) := \{  p \in CX^{\infty} \mid \langle u, p \rangle \leq h(u)\ (u \in X^{\infty} \subseteq CX^{\infty})\}.
\end{equation}
Since $p \mapsto \langle u, p \rangle$ is continuous (Lemma~\ref{lem:<p,q>}), 
we have:
\begin{Lem}\label{lem:closed}
	$B(h)$ is a closed subset in $CX^{\infty}$.
\end{Lem}
We mainly consider $B(f^{\infty})$ for recession function $f^{\infty}$.
Belonging to $B(f^{\infty})$ is a necessary condition 
for $p \in CX^{\infty}$ for which $f+b_p$ is bounded below.
\begin{Lem}\label{lem:unbounded}
$\dom f^* \subseteq B(f^{\infty})$. 
\end{Lem}
\begin{proof}
Suppose that $\langle u,p \rangle > f^{\infty}(u)$ for some $u \in X^{\infty}$.
Then $f^{\infty}(u) - \langle u,p \rangle = (f + b_p)^{\infty}(u) =  
\lim_{t \to \infty} (f(c(t)) + b_{p} (c(t)))/t$ is 
negative, where $c$ is a ray with $c(\infty) = u$.  
Therefore, for some $\alpha > 0$, it holds
$
f(c(t)) + b_{p}(c(t)) < - \alpha t
$ 
for every large $t> 0$.
This means that $\inf_{x \in X} f(x) + b_{p}(x) = - \infty$, 
and $p \not \in \dom f^*$.
\end{proof}
For a subset $R \subseteq CX^{\infty}$, let $\overline{R}$ 
denote the closure of $R$ with respect to the $d^{\infty}$-topology.
By the above lemma, it holds $\overline{\dom f^*} \subseteq B(f^{\infty})$.
We do not know whether the equality holds in general.

\subsection{Hadamard manifolds}
From here, we 
restrict our study to smooth convex optimization on a {\em Hadamard manifold}, i.e., a simply-connected complete Riemannian manifold 
having nonpositive sectional curvature.
We utilize elementary concepts in Riemannian geometry; see e.g.,~\cite{Sakai}.
A recent book~\cite{Boumal} for optimization perspectives is also useful.
For Hadamard manifolds, we consult~\cite{BGS1985,Eberlien}. 

Let $M$ be an $n$-dimensional Hadamard manifold.  
For $x \in M$, let $\langle, \rangle_x$ denote the inner product of the tangent space $T_x$. Let $d$ denote the distance function on $M$ 
obtained from the Riemannian connection.
The metric space $M$ is known to be a Hadamard space. 
A geodesic ray issuing from $x$ 
is given by the {\em exponential map} $\exp_x: T_x \to M$.
Namely, for $v \in T_x$, the map $t \mapsto \exp_x (vt)$ is 
a constant-speed ray 
with speed $\|v\|_x := \sqrt{\langle v,v \rangle_x}$.
The map $\exp_x$ is a diffeomorphism from $T_x$ to $M$.

Consider the boundary $M^{\infty}$ and its Euclidean cone $CM^{\infty}$ of $M$.
Via the exponential map, the boundary $M^{\infty}$ is identified
with the unit sphere at $T_{x}$, and the 
cone $CM^{\infty}$ is identified with $T_x \simeq \RR^n$.
This identification gives another topology to $M^{\infty}$ and to $CM^{\infty}$, which is called the {\em standard topology}, and  is independent of the choice of $x$.
In the standard topology, $M^{\infty}$ and $CM^{\infty}$ are homeomorphic to sphere $S^{n-1}$ and Euclidean space $\RR^n$, respectively.
It is known \cite[II.9.7(1)]{BrHa} that the identity map on 
$CM^{\infty}$ from the $d^{\infty}$-topology 
to the standard topology is continuous (and is not homeomorphic in general).
\begin{figure}[t]
	\begin{center}
		\includegraphics[scale=0.7]{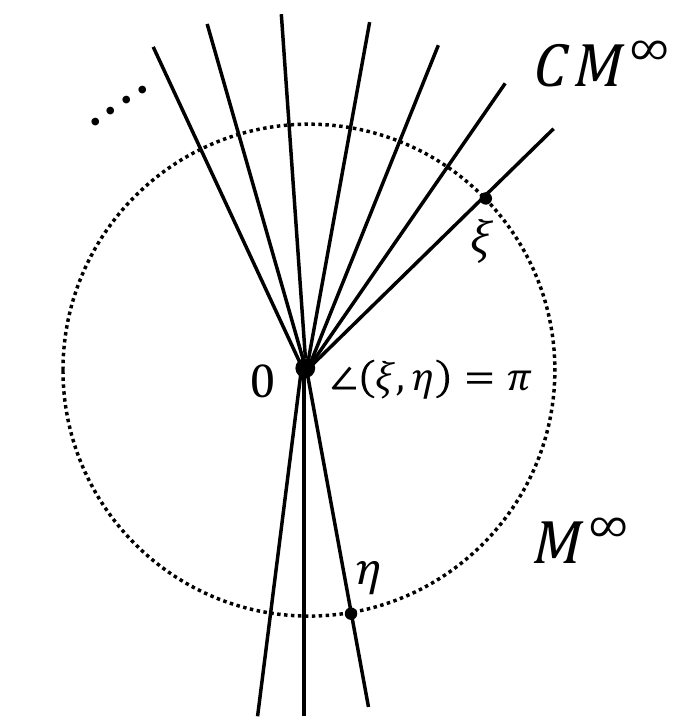}
		\caption{The boundary of hyperbolic space}
		\label{fig:hyperbolic}
	\end{center}
\end{figure}\noindent
\begin{Ex}\label{ex:hyperbolic}
	Suppose that $M$ is a hyperbolic space.
	For distinct $\xi,\eta \in M^{\infty}$ 
	there is a geodesic line $c: \RR \to M$ such that $c(\infty) = \xi$ and $c(-\infty) = \eta$. This means that 
	$\angle (\xi, \eta) = \pi$ for every distinct $\xi,\eta \in M^{\infty}$.
	Therefore, the boundary $M^{\infty}$ is a discrete topological space. See \cite[II.9.6 (ii)]{BrHa}.
	The cone $CM^{\infty}$ is a {\em star} obtained from 
	infinitely many half-lines $\RR_+$, each associated with $\xi \in M^{\infty}$, 
	by identifying the origin of all $\RR_+$. 
	See Figure~\ref{fig:hyperbolic}.
	This can be seen from the distance formula (\ref{eqn:d^infty}) as $d^{\infty}(r\xi,s \eta) = r+s$ if $\xi \neq \eta$ 
	and $|r-s|$ if $\xi = \eta$.
\end{Ex}

Let $f: M \to \RR$ be a smooth convex function. 
\begin{Lem}\label{lem:standard}
\begin{itemize}
\item[(1)] $f^{\infty}$ 
is lower semicontinuous in the standard topology~\cite[Lemma 2.11]{KL2006}.
\item[(2)] 
$(p,q) \mapsto \langle p, q \rangle$ is upper semicontinuous in the standard topology.
\end{itemize}
\end{Lem}
\begin{proof}
(1). Take an arbitrary $x \in M$, and identify $CM^{\infty}$ with $T_x$ (by $v \mapsto \exp_x v \infty$).
	Let $f_t (p) := (f(\exp_x (pt)) - f(x))/t$ for $p \in T_x$, 
which
is monotone nondecreasing in $t$ and continuous on $T_x$ (in the standard topology).
Therefore $f^{\infty}(p) = \sup_{t\geq 0} f_t (p)$. 
It is well-known that 
the supremum of continuous functions $f_t$ is lower semicontinuous.

(2). It is known \cite[II.9.5]{BrHa} that
$(\xi,\eta) \mapsto \angle (\xi, \eta)$ is 
lower semicontinuous. Then $(p,q) \mapsto \langle p,q\rangle$ 
is upper semicontinuous, since $p \mapsto d^{\infty}(0,p) = \| p \|_x$ is continuous in the standard topology and cosine is a decreasing function.  
\end{proof}

As in the Euclidean case, a minimizer of $f$ 
is characterized by the gradient.
The {\em gradient} $\nabla f(x) \in T_x$ of $f$ at $x$ is defined via
\begin{equation*}
\langle \nabla f(x), v \rangle_x = df(x)(v) = \lim_{t \to 0} \{ f(\exp_x t v) - f(x) \}/t \quad (v \in T_x),
\end{equation*} 
where $df(x):T_x \to \RR$ is the differential of $f$ at $x$.
It is easy to see:
\begin{Lem}[{see \cite[Corollary 11.22]{Boumal}}]
	$x \in M$ is a minimizer of $f$ if and only if $\nabla f(x) = 0$.
\end{Lem}
In Euclidean case $M=\RR^n$, 
$x$ is a minimizer of $x \mapsto f(x) - \langle p,x \rangle$
if and only if $\nabla f(x) = p$.
To extend it, we consider
the gradient of Busemann functions.
Note that 
Busemann functions are continuously differentiable; see \cite[1.10.2 (1)]{Eberlien}.
As in the previous subsection, we fix $x_0 \in M$ and 
identify $CM^{\infty}$ with 
the space of Busemann functions for rays issuing from $x_0$.

\begin{Lem}[{see \cite[1.10.2 (2)]{Eberlien}}]\label{lem:nabla_Busemann}
	For $x \in M$ and $p \in CM^{\infty}$, it holds
	$\nabla b_{p}(x) = - u$ for 
	$u \in T_x$ with $p = \exp_x \infty u$.
\end{Lem}

The {\em asymptotic gradient} $\nabla^{\infty} f(x)$ for $x \in M$  is defined as the asymptotic class of 
the ray $t \mapsto \exp_x t \nabla f(x)$, that is,
\begin{equation*}
\nabla^{\infty} f(x) := \exp_x \infty \nabla f(x). 
\end{equation*}
Notice that $\|\nabla^{\infty}f(x)\| = d^{\infty}(0, \nabla^{\infty}f(x))$ is the speed of $t \mapsto \exp_x t \nabla f(x)$. Therefore we have
\begin{equation}\label{eqn:d^infty(0,nabla)}
\|\nabla^{\infty}f(x)\| = \|\nabla f(x)\|_x. 
\end{equation}
Then we have the following analogue of the one in Euclidean convex analysis.
\begin{Lem}\label{lem:f(x)+f*(p)=-b_p(x)}
For $p \in CM^{\infty}$, 
the following conditions are equivalent:
\begin{itemize}
\item[(i)] $x$ is a minimizer of $f + b_{p}$ over $M$.
\item[(ii)] $\nabla^{\infty} f(x) = p$.
\item[(iii)] $f(x) + f^*(p) = - b_{p}(x)$
\end{itemize}
\end{Lem}
\begin{proof}
	(i) $\Leftrightarrow$ 
	$\nabla (f + b_{p})(x) = 0$ 
	$\Leftrightarrow$ $\nabla f(x) = u$ for 
	$p = \exp_x u \infty$
	 $\Leftrightarrow$ (ii). 
	 (i) $\Leftrightarrow$ (iii) is obvious from the definition~(\ref{eqn:f*}) of conjugate $f^*$. 
\end{proof}
As in the Euclidean case, 
the conjugate of a (smooth) convex function 
recovers the original function 
via the inverse transformation.
\begin{Lem}
$
f(x) = \sup_{p \in CM^{\infty}} - b_{p}(x) - f^*(p) \quad (x \in M).
$
\end{Lem}
\begin{proof}
	By definition, $f^{*}(p) \geq - b_{p}(x) - f(x)$ 
	for every $x \in M$ and $p \in CM^{\infty}$.
	Therefore $f(x) \geq \sup_{p \in CM^{\infty}} - b_{p}(x) - f^{*}(p)$. 
	For $p = \nabla^{\infty} f(x)$, 
	the equality is attained by Lemma~\ref{lem:f(x)+f*(p)=-b_p(x)}.	
\end{proof}
This gives rise to an interesting question of characterizing the class of functions $g$ 
on $CM^{\infty}$ for which $g^*(x) := \sup_{p \in CM^{\infty}} - b_{p}(x) - g(p)$
is convex on $M$. However this is beyond the theme of this paper, 
and we leave it for future research.

For a set $R \subseteq CM^{\infty}$, 
let $\interior R$ denote the interior of $R$ 
in the $d^{\infty}$-topology.
\begin{Lem}\label{lem:interior}
	Let $h: CM^{\infty} \to \RR \cup \{\infty\}$ be a positively homogeneous function.
	Suppose that $h$ is lower semicontinuous in the standard topology.
	Then it holds
	\begin{equation}\label{eqn:int_B(h)}
\interior B(h) 
= \{ p \in CM^{\infty} \mid  \langle \xi, p \rangle < h(\xi)\ (\xi \in M^{\infty} \subseteq CM^{\infty})\}.
\end{equation}
\end{Lem}
\begin{proof} Let $p \in B(h)$.
Suppose that $\langle \xi,p \rangle  = h(\xi)$ for some $\xi \in M^{\infty}$.
Then, for arbitrary $\epsilon > 0$,
$\langle \xi,p+ \epsilon \xi \rangle \geq \langle \xi,p \rangle + \epsilon \|\xi\|^2 > h(\xi)$ (Lemma~\ref{lem:<p,q>}). 
This means $p+ \epsilon \xi \not \in B(h)$.
Since $\langle,\rangle$ is continuous (Lemma~\ref{lem:<p,q>}), $p$ is never an interior point of $B(h)$. 

Suppose that $\langle \xi,p \rangle  < h(\xi)$ for all $\xi \in M^{\infty}$.
For $q \in CM^{\infty}$, 
let $\Delta (q):= \inf_{\xi \in M^{\infty}} h(\xi) - \langle \xi, q \rangle$.
Since $M^{\infty}$ is compact and $\xi \mapsto h(\xi) - \langle \xi, q \rangle$
is lower semicontinuous in the standard topology (Lemma~\ref{lem:standard}~(2)),
the infimum is always attained.
Let $\epsilon := \Delta(p) > 0$.
Since $(q,\xi) \mapsto h(\xi) -\langle \xi,q\rangle$ is lower semicontinuous, 
for each $\xi \in M^{\infty}$ there is an open neighborhood $U_{\xi} \times V_{\xi}$
of $(p,\xi)$ (in the standard topology)
such that for each $(q,\eta) \in U_{\xi} \times V_{\xi}$ we have
$h(\eta) -\langle \eta, q\rangle \geq h(\xi) -\langle \xi, p\rangle - \epsilon \geq \Delta(p) - \epsilon = 0$.
Since $M^{\infty}$ is compact, there are $\xi_1,\xi_2,\ldots,\xi_m$ 
with $M^{\infty} = \bigcup_{i=1}^m V_{\xi_i}$.
Let $U := \bigcap_{i=1}^m U_{\xi_i}$, 
which is an open neighborhood of $p$ (in the standard topology).
For any $q \in U$, 
$\Delta(q) = h(\xi') - \langle \xi',q \rangle$ for some $\xi'$.
Since $\xi'$ belongs to $V_{\xi_i}$ for some $i$, 
it holds $\Delta(q) = h(\xi') - \langle \xi',q \rangle \geq 0$.
Hence, we have $p \in U \subseteq B(h)$.
Since the identify map on $CM^{\infty}$ 
from the $d^{\infty}$-topology to the standard topology 
is continuous, 
$U$ is an open neighborhood of $p$ in $d^{\infty}$-topology. 
\end{proof}
As the proof shows, the RHS of (\ref{eqn:int_B(h)}) is the interior of $B(h)$ 
also in the standard topology, although $B(h)$ may not be closed in this topology.

We apply this lemma to the recession function $f^{\infty}$  and the associated subset $B(f^{\infty})$.
\begin{Prop}\label{prop:minimizer} 
For $p \in \interior  B(f^{\infty})$, there is a minimizer of 
$f + b_{p}$. In particular, any point in $B(f^{\infty}) \setminus \dom f^*$ belongs to the boundary of $B(f^{\infty})$.
\end{Prop}
See \cite[Lemma 3.2 (iv)]{KLM2009JDG} for a related argument.
\begin{proof}
Since $\xi \mapsto f^{\infty}(\xi) - \langle \xi,p \rangle$ is 
lower semicontinuous on 
compact set $M^{\infty}$ in the standard topology,
the minimum value $\alpha > 0$ exists. 
Let $\alpha' \in (0, \alpha)$ and let $g := f+ b_p$.
Fix an arbitrary $x \in M$.
Then, for every $\xi \in M^{\infty}$ there is $t_{\xi} \in \RR_+$ 
such that $g (\exp_x t \xi) - g (x) > \alpha' t$ for all $t \geq t_{\xi}$.
Define $h: M^{\infty} \to \RR_+$ by
\begin{equation*}
h(\xi) := \inf \{ t \geq 0 \mid g (\exp_x t \xi) - g (x) > \alpha' t \} \quad (\xi \in M^{\infty}).
\end{equation*}
Then $h$ is upper semicontinuous, 
since the epigraph $\{ (\xi,t) \in M^{\infty} \times \RR \mid t \leq h(\xi) \}$ is the closed set
$\{ (\xi,t) \in M^{\infty} \times \RR_+ \mid g (\exp_x t \xi) - g (x) \leq \alpha' t \} 
\cup \{ (\xi,t) \in M^{\infty} \times \RR \mid t \leq 0\}$.
Since $M^{\infty}$ is compact, 
the maximum $t^*$ of $h$ over $M^{\infty}$ exists.
Then, for every  $\xi \in M^{\infty}$, we have
$
g (\exp_x t \xi) - g (x) \geq \alpha' t$ for all $t \geq t^*$.
This means that the level set
$\{y \in M \mid g (y) \leq g (x) + \alpha' t^*\}$ belongs to 
the metric ball at center $x$ with radius $t^*$, which is compact by Hopf-Rinow theorem (see \cite[III.1]{Sakai}). 
A minimizer of $g$ exists in this set.  
\end{proof}
Thus we have
\begin{equation}\label{eqn:inclusion}
\interior B(f^{\infty}) \subseteq \nabla^{\infty}f(M) 
\subseteq \dom f^* \subseteq B(f^{\infty}).
\end{equation}
%

%
We next provide a characterization of $B(f^{\infty})$, which sharpens the following important result by Kapovich, Leeb, and Millson~\cite{KLM2009JDG}.
\begin{Thm}[{\cite[Lemma 3.4]{KLM2009JDG}}]
	If $f^{\infty}(u) \geq 0$ for all $u \in M^{\infty}$, i.e., $0 \in B(f^{\infty})$,
	then $\inf_{x \in M}\|\nabla f(x)\|_{x} = 0$. 
\end{Thm}
An outline of the proof is as follows: 
If $\inf_{x \in M}\|\nabla f(x)\|_x >0$, then
a trajectory of the normalized gradient flow of $f$ goes to $u \in M^{\infty}$ with $f^{\infty}(u) < 0$; See also \cite[section 5.4]{Woodward}.

We now obtain an analogue of the equivalence between (a) and (c) in the introduction.
\begin{Thm}~\label{thm:char_of_B(f^infty)}
	For $p \in CM^{\infty}$, 
	the following are equivalent:
	\begin{itemize}
		\item[(a)] $\inf_{x \in M}\|\nabla (f+b_p)(x) \|_{x} = 0$.
		\item[(c)] $p \in B(f^{\infty})$.  
	\end{itemize}	
\end{Thm}
%
\begin{proof}
	(c) $\Rightarrow$ (a) follows from applying the above theorem to $f+b_p$.
	We verify (a) $\Rightarrow$ (c). Let $g := f+b_p$.
	For arbitrary $\epsilon > 0$, there is $x \in M$ 
	such that $\|\nabla g(x)\|_x < \epsilon$.
	Consider $u \in M^{\infty}$ and 
	the geodesic ray $t \mapsto \exp_x v t$ with $\exp_x v \infty= u$. 
	Then $\lim_{t \to 0}  (g(\exp_x v t) - g(x))/t  = \langle \nabla g(x), v \rangle_x  = \|\nabla g(x)\|_x \cos \theta > - \epsilon$, where $\theta$ is the angle between $\nabla g(x)$ and $v$ in $T_x$.
    Since $(g(\exp_x v t) - g(x))/t$ is monotone nondecreasing, by $t \to \infty$ 
    we have $g^{\infty}(u) = f^{\infty}(u) - \langle u,p \rangle > - \epsilon$.
	Thus, for every $\epsilon > 0$ and $u \in M^{\infty}$, 
	it holds $\langle u,p \rangle < f^{\infty}(u) + \epsilon$.
	This implies $p \in B(f^{\infty})$.
\end{proof}
The condition (a) $\inf_{x\in M}\|\nabla (f+b_p)(x) \|_{x} = 0$ 
may be viewed as a correspondent of $\inf_{x \in \RR^n} \|\nabla f(x) - p\| = 0$ 
of the Euclidean case.
If $\nabla^{\infty} f(x_i)$ $(i=1,2,\ldots)$ converges to $p \in CM^{\infty}$ 
in the $d^{\infty}$-topology, then (a) holds.
Indeed, by Lemma~\ref{lem:nabla_Busemann} and (\ref{eqn:d^infty(0,nabla)}) it holds
\begin{eqnarray*}
	\|\nabla (f + b_p)(x) \|_{x}^2 &= &\|\nabla f(x)\|_x^2 + \|u\|_x^2 - 2 \|\nabla f(x)\|_x 
	\|u\|_x \cos \theta_x \\
	&\leq &\|\nabla^{\infty}f(x)\|^2 + \|p\|^2 -  2 \|\nabla^{\infty}f(x)\|\|p\| \cos \theta \\
	&= & d^{\infty}(\nabla^{\infty}f(x),p)^2, 
\end{eqnarray*}
where $u := - \nabla b_p(x)$, $\theta_x$ is the angle between $\nabla f(x)$ and $u$ in $T_x$, 
and $\theta := \angle (\nabla^{\infty}f(x), p)$ ($\geq \theta_x$ by definition (\ref{eqn:angle})).  
However, the converse is not true. 
Also (a) does not mean the convergence in the standard topology. 

%

\subsection{Symmetric spaces of nonpositive curvature}\label{subsec:symmetric_space}

Here we consider symmetric spaces of nonpositive curvature, 
which constitute a fundamental class of Hadamard manifolds.
Our argument basically consults~\cite[Chapters 2 and 3]{Eberlien}. 
%
In a Hadamard manifold $M$ and a point $x$, 
the {\em geodesic symmetry} $\sigma_x:M \to M$ at $x$
is defined by $\sigma_x(\exp_x(v)) = \exp_x(-v)$ for $v \in T_x$.
A {\em symmetric space of nonpositive curvature} is a Hadamard manifold $M$ 
such that for every $x \in M$ the geodesic symmetry $\sigma_x$ is an isometry on $M$.
Via the de Rham decomposition theorem (see \cite[III.6]{Sakai}), $M$ is 
(uniquely) decomposed as Riemannian product $M = M_0 \times N$, 
where $M_0$ is isometric to Euclidean space $\RR^k$ $(k \geq 0)$  
(called the Euclidean de Rham factor of $M$) 
and  a symmetric space $N$ {\em of noncompact type},  i.e.,  it is given
by $N = G/K$ for a (real) semisimple Lie group $G$ and its maximal compact subgroup $K$.
Then $N$ has a trivial Euclidean de Rham factor, and  
its Riemannian structure is given by a $G$-invariant metric.
We will see more concrete constructions in Sections~\ref{subsec:P_n} and \ref{subsec:orbit}.
%

Let $M$ be a symmetric space of nonpositive curvature, and let $x_0 \in M$.
A $k$-dimensional {\em flat} is a submanifold of $M$ isometric to $\RR^k$. 
A {\em maximal flat} is a flat that is not contained in another flat of a larger dimension.
It is a basic fact that all maximal flats have the same dimension $d$, 
which is called the {\em rank} of $M$.
By a {\em geodesic line} 
we mean a map $l: \RR  \to M$
with $d(l(s),l(t)) = |s - t|$ for $s,t \in \RR$, which is just a $1$-dimensional flat.
A geodesic line is called {\em regular} if it is contained by a unique maximal flat.
Let $x \in M$ and let $F$ be a maximal flat containing $x$. 
A {\em Weyl chamber} at tip $x$ is 
the closure of a connected component of the set of points $y \in F \setminus \{x\}$ 
such that the unique geodesic line containing $x,y$ is regular.
When $F$ is viewed as $\RR^d$ with origin $x$, 
Weyl chambers are polyhedral cones.
They have the same shape, since
the group $G$ acts transitively on them.

We next explain the building structure of 
the boundary $M^{\infty}$ and its cone $CM^{\infty}$;
\cite[Appendix 5]{BGS1985} is a useful reference. 
It is clear (from Example~\ref{ex:Euclidean_X^infty}) that the boundary $F^{\infty}$ 
of a maximal flat $F$ is isometric to sphere $S^{d-1}$. 
The boundary $C^{\infty}$ of a Weyl chamber $C$ (at some tip)
is called an (asymptotic) {\em Weyl chamber}. 
For two Weyl chambers $C,D$ (at possibly different tips), 
the asymptotic ones $C^{\infty}, D^{\infty}$ are the same or have disjoint interiors.
Then, the set of all Weyl chambers and their faces 
give rise to a cell-complex structure on $M^{\infty}$, where 
each cell is isometric to a polyhedral cell in a sphere (the intersection of a sphere and a polyhedral cone). 
This structure is an {\em $M_{1}$-polyhedral complex} in the sense of \cite[Chapter I.7]{BrHa}, 
and forms a {\em spherical building}, 
where an {\em apartment} is precisely the subcomplex formed by the boundary of 
a maximal flat. 
Although Weyl chambers here may not be simplices, 
one can subdivide them to obtain a simplicial complex ${\cal C}$ so that each apartment is a spherical Coxeter complex. Then the apartments are glued nicely. That is, 
they satisfy, as an abstract simplicial complex, 
the axiom of building:
\begin{itemize}
	\item Any two simplices in ${\cal C}$ are contained in a common apartment.
	\item For two apartments $A,A'$ 
	including simplices $C,C'$, 
	there is an isomorphism $A\to A'$ fixing $C,C'$ pointwise.
\end{itemize}
See e.g., \cite{BuildingBook} and \cite[II.10. Appendix]{BrHa} for (formal) theory of building.
By considering the Euclidean cone,  
$CM^{\infty}$ has 
the structure of a {\em Euclidean building}.   
A maximal cone is also called a Weyl chamber.
Apartments are the subcomplexes induced by $CF^{\infty}$ for all maximal flats $F$.
We simply call $CF^{\infty}$ an apartment.
Each apartment is a convex subspace of $CM^{\infty}$ 
isometric to $\RR^d$. 
We can identify an apartment $E \subseteq CM^{\infty}$ with $\RR^d$ so that the origins coincide. 
In this identification,  
$\langle, \rangle$ in $E$ is precisely the Euclidean inner product of $\RR^d$.
%
%

Suppose that $M = \RR^k \times N$, where $N$ is a symmetric space $G/K$ of noncompact type.
Then $CM^{\infty} = \RR^k \times CN^{\infty}$. 
The structure of building $CM^{\infty}$ 
is determined by $CN^{\infty}$.
We can suppose (as in \cite{Eberlien}) that $G$ is the identity component of 
the group of isometries of $N$, where $G$ acts isometrically on $M = \RR^k \times N$ 
by $\RR^k \times N \ni (x',x) \mapsto (x',gx)$.
Since any isometry on $M$ induces an isometry on $M^{\infty}$, 
the group $G$ acts isometrically on $M^{\infty}$ and 
on $CM^{\infty}$ by $g p = \|p\| g (p/\|p\|)$. 
Also $G$ acts on the set of Weyl chambers and their faces. 
The facial incidence structure of the building $M^{\infty}$ is 
described by the inclusion relation of all parabolic subgroups of $G$. 
Here a {\em parabolic subgroup} is the subgroup consisting of 
$g \in G$ with $g \xi = \xi$ for some $\xi \in N^{\infty}$, 
which is denoted by $G_\xi$.
If $\xi, \xi'$ belong to 
the relative interior of the same face of a Weyl chamber, then $G_{\xi} = G_{\xi'}$.
Therefore, 
minimal parabolic subgroups
correspond to Weyl chambers.
Fix a Weyl chamber $C_0$, and regard it as a polyhedral cone in $\RR^d$.
Let $P$ be the minimal parabolic subgroup for $C_0$. 
The set of Weyl chambers is identified with the flag variety $G/P$. 
If $p \in CM^{\infty}$ belongs to a Weyl chamber corresponding to ${\cal F} \in G/P$ 
and the orbit of $p \in CM^{\infty}$ by $G$-action 
meets a (unique) point $\lambda$ in $C_0 \subseteq \RR^d$, 
then we denote $p$ by
\begin{equation}\label{eqn:lambda_F}
p = \lambda \cdot {\cal F}.
\end{equation}   
This notation  
designates the coordinate $\lambda$ of $p$ in the Weyl chamber indexed by ${\cal F}$. 

Two geodesic lines $c,c':\RR \to M$ are said to be {\em parallel} 
if $d(c(t),c'(t))$ $(t \in \RR)$ is bounded above.
For a geodesic line $c$, let $F(c)$ denote the union of all geodesic lines parallel to $c$.
It is known that $F(c)$ is a totally geodesic submanifold of $M$.
If $c$ is regular, $F(c)$ is a maximal flat.

For $\xi \in N^{\infty}$, the {\em horospherical subgroup} $N_{\xi}$ of $G_{\xi}$
consists of $n \in G$ such that $\lim_{t \to \infty} d(n(c(t)), c(t)) = 0$, 
where $c$ is the geodesic ray with $c(\infty) = \xi$ and $c(0) = x_0$ (independent of $x_0$).
Then  $N_\xi$ keeps 
the Busemann function $b_\xi$ as $b_\xi (n x)= b_\xi (x)$, since 
$|b_\xi (n x) - b_{\xi}(x)|= \lim_{t \to \infty} |d(x, n^{-1}c(t)) -d(x,c(t))| 
\leq \lim_{t \to \infty} d(c(t), n^{-1}c(t)) = 0$.
The {\em generalized Iwasawa decomposition}~\cite[2.17.5 (5)]{Eberlien} implies that $M$ 
is diffeomorphic to $N_{\xi} \times F(c)$ by $(n,y) \mapsto n(y)$. 

In this setting, let us start our convex analysis on $M$.
Let $f: M \to \RR$ be a smooth convex function.
For a maximal flat $F$, 
let $f_F: F \to \RR$ denote the restriction of $f$ to $F$.
The asymptotic Legendre-Fenchel conjugate ${f_F}^*: CF^{\infty} \to \RR \cup \{\infty\}$ 
is defined by ${f_F}^*(p):= \sup_{x \in F} - b_{p}(x) - f(x)$, 
where $x_0$ may be outside of~$F$.
By $p \in CF^{\infty}$ and Example~\ref{ex:Euclidean_Busemann}, 
$b_{p}$ is an affine function on $F$.
Therefore, ${f_F}^*$
is viewed as the ordinary Legendre-Fenchel conjugate  (up to an additive constant)
under identification $CF^{\infty} = \RR^d$.  Hence we have:
\begin{Lem} For any maximal flat $F$, 
	${f_F}^*$ is a convex function on $CF^{\infty}$.
\end{Lem}
Still, $f^{*}$ may be nonconvex but ``partially" convex in the following sense:
\begin{Prop}\label{prop:partial_convexity}
	For a Weyl chamber $C \subseteq CM^{\infty}$, it holds
	\begin{equation}
	f^*(p) =\sup_{F} {f_F}^* (p) \quad (p \in C),
	\end{equation}
where $\sup$ is taken over all maximal flats $F$ 
with $CF^{\infty} \supseteq C$.
In particular, $f^*$ is a convex function on $C$, and 
 $C \cap \dom f^*$ is a convex set contained by a convex set $\bigcap_F C \cap \dom {f_F}^*$:
 \begin{equation}\label{eqn:subseteq}
 		C \cap \dom f^* \subseteq \bigcap_F C \cap \dom {f_F}^*.
 \end{equation}
\end{Prop}
Note that $\subseteq$ may be strict since
$f^*(p) = \sup_F {f_F}^*(p)$ may be $\infty$ 
even if ${f_F}^*(p) < \infty$ for every $F$. 
\begin{proof}
Consider the Iwasawa decomposition $M = N \times F_0$, 
where $F_0$ is a maximum flat containing $C$ at infinity and 
$N$ is the holospherical subgroup for $C$.     
Then $NF_0$ ranges over all maximal flats containing $C$ at infinity.
Thus
$M$ is the (disjoint) union  
of all maximal flats $F$ such that $F^{\infty} \supseteq C \ni p$.
Then we have $f^*(p) = \sup_{F} \sup_{x \in F} - b_p (x) - f(x) =  \sup_F {f_F}^*(p)$.
\end{proof}
%


Next, we consider the recession function $f^{\infty}$ and the associated subset 
$B(f^{\infty})$.
For each apartment $E$, we also consider
\begin{equation}
B_E(f^{\infty}) := \{ p \in E \mid \langle u,p \rangle \leq f^{\infty}(u)\ (u \in E)\}.
\end{equation}
\begin{Prop}\label{prop:C_cap_B(h)}
	For a Weyl chamber $C \subseteq E$, it holds
	\begin{equation}\label{eqn:C_cap_B(f^infty)}
	C  \cap  B(f^{\infty})  = \bigcap_{F} C \cap B_{CF^{\infty}} (f^{\infty})  = \bigcap_F C \cap \overline{\dom {f_F}^*}.
	\end{equation}
	where $\bigcap$ is taken over all maximal flats $F$ with $CF^{\infty} \supseteq C$.
In particular, $C \cap B(f^{\infty})$ is a convex set in $C$.
\end{Prop}
\begin{proof}
Since $CM^{\infty}$ is the union of all apartments $E$ containing any fixed Weyl chamber $C$, 
we have the first equality.
When $E$ is viewed as Euclidean space $\RR^n$, then $C$ is a convex cone, 
$\langle u,p \rangle \leq f^{\infty}(u)$ is a linear inequality, 
and therefore $C \cap B_E(f^{\infty})$ is convex.
Necessarily the intersection  $\bigcap_{E} C \cap B_E (f^{\infty})$ 
is convex.
By Lemma~\ref{lem:standard}, 
the restriction of $f^{\infty}$ to Euclidean space $E$ 
is lower semicontinuous and positively homogeneous convex, and must be the support function of $\dom {f_F}^*$.
Hence we have $\overline{\dom {f_F}^*} = B_{E}({f}^{\infty})$, and the second equality. 
\end{proof}
In the view of (\ref{eqn:subseteq}) and (\ref{eqn:C_cap_B(f^infty)}), 
the closure $\overline{\dom f^*}$
seems very close to $B(f^\infty)$, 
although we do not know whether they really differ.
Under non-degeneracy assumptions, three spaces $\nabla^{\infty} f(M)$, $\overline{\dom f^*}$, and
$B(f^{\infty})$ are equal, as in Euclidean case.
\begin{Prop}\label{prop:main0}
	For a Weyl chamber $C$, 
	if $C \cap B(f^{\infty})$ has nonempty interior, then
	$C \cap \overline{\nabla^{\infty} f(M)} = C \cap \overline{\dom f^{*}} = C \cap B(f^{\infty})$.
\end{Prop}
\begin{proof}
By (\ref{eqn:inclusion}), it holds 
\begin{equation}\label{eqn:inclusion2}
C \cap \interior B(f^{\infty}) 
\subseteq C \cap \overline{\nabla^{\infty} f(M)} \subseteq C \cap \overline{\dom f^{*}} \subseteq C \cap B(f^{\infty}).
\end{equation}	
In Euclidean space,
for any closed convex set $D$ having nonempty interior, 
it holds $\overline{\interior D} = D$. 
This can be applied to $C \cap B(f^{\infty})$, which is viewed as a Euclidean convex set.
From  $C \cap B(f^{\infty}) = \overline{\interior (C \cap B(f^{\infty}))} \subseteq \overline{C \cap \interior B(f^{\infty})} 
\subseteq C \cap B(f^{\infty})$, by taking the closure in (\ref{eqn:inclusion2}), we have the claim. 
\end{proof}
\begin{Prop}
	Suppose that $f$ is strictly convex.
	Then the gradient map $x \mapsto \nabla^{\infty}f(x)$ is  
	a bijection from $M$ to $\interior B(f^{\infty})$.  
	For a Weyl chamber $C$, if $C \cap \nabla^{\infty}f(M) \neq \emptyset$, 
	then $C \cap \overline{\nabla^{\infty}f(M)} =C \cap \overline{\dom f^*} = C \cap B(f^{\infty})$.
\end{Prop}
\begin{proof}
	We first verify $\nabla^{\infty}f: M \to CM^{\infty}$ is injective.
	Indeed, if $\nabla^{\infty}f(x) = \nabla^{\infty}f(y) = p$, 
	then $x$ and $y$ are minimizers of $f + b_p$, and 
	it must hold $x=y$ by strict convexity of $f + b_p$.
	For surjectivity, 
	from Proposition~\ref{prop:minimizer} 
	it suffices to show that $\nabla^{\infty}f(x)$ for any $x \in M$ 
	belongs to $\interior B(f^{\infty})$. We utilize the following property, 
	where $C$ is a Weyl chamber.
	\begin{equation}\label{eqn:b_p+b_q}
	b_p+b_q =  b_{p+q} \quad (p,q \in C).
	\end{equation}
	Indeed, consider a maximal flat $F$ containing $x_0$ and $C$ at infinity, and
	consider horospherical subgroup $N$ of $C$ and Iwasawa decomposition $M = N \times F$.
	If $x = n(y)$ for $n \in N$ and $y \in F$, 
	then $b_p(x) + b_q (x) = b_p(y) + b_q (y) 
	= - \langle p, y-x_{0} \rangle - \langle q, y-x_{0} \rangle =
	-\langle p+q, y-x_{0} \rangle = b_{p+q}(y) = b_{p+q}(x)$, 
	where $p+q \in C$ and $F$ is viewed as $\RR^d$ (see Example~\ref{ex:Euclidean_Busemann}).
	
	Let $p := \nabla^{\infty}f(x)$ for $x \in M$. 
	Suppose that $0 \neq p = \|p\| \xi$ for $\xi \in M^{\infty}$.
	Consider sufficiently small $\epsilon > 0$.
	Let $B_{\epsilon}$ be the set of all points $q$ represented as 
	$q = (\|p\| - t) \xi + 2 t \eta$, where
	$t \leq \epsilon$ and 
	$\xi, \eta \in CM^{\infty}$ belong to the same Weyl chamber.  
    Then $B_{\epsilon}$ contains $p$ in its interior.
	
	We show that $B_{\epsilon} \subseteq \nabla^{\infty}f(M)$ for small $\epsilon > 0$.
	By strict convexity of $f$, 
	$x$ is a unique minimizer of $f + b_p$.
	Let $\alpha > 0$ be the minimum of $(f+ b_p) (\exp_x v) - (f+b_p)(x)$
	over all $v \in T_x$ with $\|v\|_x =1$, 
	which exists by compactness.
	For $q  = (\|p\| - t) \xi + 2 t \eta \in B_{\epsilon}$, 
	it holds $f+b_q = f+ b_p - t b_{\xi} + 2t b_{\eta}$ by (\ref{eqn:b_p+b_q}). 
	Here the additional term $- t b_{\xi} + 2t b_{\eta}$ 
	is $3\epsilon$-Lipschitz (Lemma~\ref{lem:Busemann_convex}).
	For $3 \epsilon \leq \alpha$, 
	a minimizer of $f+b_q$ exists in the unit ball around $x$.
	Thus $q \in \nabla^{\infty}f(M)$, as required.
	The proof of the case $p=0$ is similar; omit $-t\xi$ above. 
	
	If $\nabla^{\infty}f(M) \cap C \neq 0$, 
	then $\nabla^{\infty}f(M) \cap C$ has nonempty interior, 
	and Proposition~\ref{prop:main0} is applicable to obtain the latter statement.
\end{proof}
Note that there is a possibility that $C \cap \nabla^{\infty}f(M)$
is empty but $C \cap B(f^{\infty})$ is nonempty and has no interior.
Further refined study on such a degenerate situation is left for future research.  

Via the inverse of $x \mapsto \nabla^{\infty}f(x)$, 
the symmetric space $M$ is coordinated by the asymptotic gradient space $\nabla^{\infty}f(M)$.
This can be viewed as a generalization of the {\em dual coordinate} of dually-flat manifolds
in information geometry~\cite{AmariNagaoka}.
It is an interesting research direction to develop an information-geometrical theory 
based on this idea.

In the case of $\dom f^* = B(f)$ which we will face in Section~\ref{sec:scaling}, 
the boundedness of $\inf_{x \in M} (f+b_p) (x)$ is verified by 
convex optimization on Euclidean building $CM^{\infty}$: 
\begin{equation}\label{eqn:optimization}
	\mbox{inf.} \quad f^{\infty}(u) - \langle u, p \rangle \quad \mbox{s.t.} \quad u \in U, 
\end{equation}
where $U \subseteq CM^{\infty}$ is any convex neighborhood 
of the origin, such as a ball.

\subsection{Symmetric space $P_n = GL(n, \CC) / U(n)$}\label{subsec:P_n}

Here we consider the symmetric space $P_n=GL(n,\CC)/U(n)$ 
of positive definite Hermitian $n \times n$ matrices, 
and present concrete descriptions and specializations 
of several concepts introduced above. 
Arguments regarding the PSD-cone as a symmetric space are found 
in \cite[Chapter II.10]{BrHa} and \cite[2.2.13]{Eberlien}, where
they consider $GL(n,\RR)/O(n)$ or $SL(n,\RR)/SO(n)$ but the arguments are analogous.     

When regarding $P_n$ as a Riemannian manifold,
the tangent space $T_x$ at $x \in P_n$ is identified with the space $S_n$ of $n \times n$ Hermitian matrices and the inner product is 
given by $\langle H, H'\rangle_x = \trace x^{-1} H x^{-1} H'$.
The cotangent space $T_x^{\ast}$ is also identified with $S_n$ 
by $H(H') := \trace H H'$ for $H,H' \in S_n$. 
Let $d$ denote the corresponding distance function on $M$.
The exponential map $\exp_x: T_x \to P_n$ at $x$ 
is given by $H \mapsto x^{1/2} e^{x^{-1/2}Hx^{-1/2}}x^{1/2}$.
In particular, any constant-speed ray is written as  
$t \mapsto g e^{t \diag \lambda} g^{\dagger}$ 
for $\lambda \in \RR^n$
and $g \in GL(n,\CC)$, where $\diag \lambda$ denotes 
the diagonal matrix with diagonal entries $\lambda_1,\lambda_2,\ldots,\lambda_n$ in order, and $(\cdot)^{\dagger}$ denotes the complex conjugate. 
By $(g,x) \mapsto gxg^{\dagger}$, 
$GL(n,\CC)$ acts isometrically and transitively on $P_n$. The isotropy group at $I$ is the group $U(n)$ of unitary matrices. 
So $P_n$ is a symmetric space $GL(n,\CC)/U(n) = \RR \times (SL(n,\CC)/SU(n))$, where 
the geodesic symmetry at $x \in P_n$ is given by $y \mapsto x y^{-1}x$.
The identity matrix $I$ is naturally chosen as 
the base point $x_0$ of $P_n$.

Any maximal flat is the set of matrices of form
\begin{equation}
F(g) := \{ g e^{\diag \lambda} g^{\dagger} \mid \lambda \in \RR^n\}
\end{equation}
for $g \in GL(n,\CC)$, where
$\lambda \mapsto g e^{\diag \lambda} g^{\dagger}$ is an isometry from $\RR^n$ to $F(g)$.
From this, we see that a geodesic ray $t \mapsto g e^{t \diag \lambda}g^{\dagger}$ 
is regular if and only if all values $\lambda_i$ are different.
A vector $\lambda \in \RR^n$ is said to be {\em arranged} 
if $\lambda_1 \geq \lambda_2 \geq \cdots \geq \lambda_n$.
The next lemma characterizes asymptotic classes of geodesic rays.
\begin{Lem}[{See \cite[II.10.64]{BrHa}}]\label{lem:asymptotic}
For arranged vectors $\lambda,\mu \in \RR^n$ with $\|\lambda\|_2=\|\mu\|_2=1$ and $g,h \in GL(n,\CC)$, 
two geodesic rays $t \mapsto g e^{t \diag \lambda} g^{\dagger}$ and  $t \mapsto h e^{t \diag \mu} h^{\dagger}$ are asymptotic if and only if  $\lambda=\mu$ and $(h^{-1}g)_{ij} = 0$ for all $i \geq j$ with $\lambda_i < \lambda_j$.
\end{Lem}
\begin{proof}
	Let $b(t) := e^{- t \diag \mu/2} h^{-1}g e^{t \diag \lambda/2}$, 
	where $b(t)_{ij} = e^{-t(\mu_i-\lambda_j)/2} (h^{-1}g)_{ij}$.
	Then we have $d(g e^{t \diag \lambda} g^{\dagger},  h e^{t \diag \mu} h^{\dagger})
	= d(b(t) b(t)^{\dagger},I)$. 
	Consider the smallest $i$ with $\lambda_i \neq \mu_i$ if it exists.
	Say $\lambda_i < \mu_i$. Then $\lim_{t \to \infty} b(t)$ 
	has a zero block of $i$ rows and $n-i+1$ columns, 
	and is singular (and/or has a $\infty$ entry).
	Necessarily $d(b(t) b(t)^{\dagger},I) \to \infty$.
	If $\lambda= \mu$ and $(h^{-1}g)_{ij} \neq 0$ for some $i \geq j$ with $\lambda_i < \lambda_j$, then $(b(t) b(t)^{\dagger})_{ii}  \to \infty$, implying $d(b(t) b(t)^{\dagger},I) \to \infty$.
	Conversely, suppose that the condition is satisfied. 
	Then $b(t)$ converges to a constant matrix.
	This means that $d(b(t) b(t)^{\dagger},I)$ is bounded, and the two rays are asymptotic.	
\end{proof}

Thus,  the minimal parabolic subgroup 
for  Weyl chamber $C_0 := \{ e^{\infty \diag \lambda} \mid \lambda:\mbox{arranged} \}$
is the group $B$ of upper triangular matrices.
Then $GL(n, \CC)/ B$ is viewed as the space of complete flags $U_1 \subset U_2 \subset \cdots \subset U_n = \CC^n$ of vector subspaces $U_i$ of $\CC^n$.
As an abstract simplicial complex, 
the building $CP_n^{\infty}$ is the order complex of 
the lattice of all (nonzero) vector subspaces of $\CC^n$. 
As consistent with (\ref{eqn:lambda_F}), 
any point $p$ in $CP^{\infty}_{n}$ is represented as $p = \lambda \cdot {\cal U}$ for
an arranged vector $\lambda$ and complete flag ${\cal U}$.
Specifically, if $p$ is written as $p = g e^{\infty \lambda} g^{\dagger}$ for an arranged vector $\lambda$ and $g \in GL(n,\CC)$, 
then ${\cal U}$ is the complete flag consisting 
of vector subspaces spanned by the first $i$ columns of $g$ 
for $i=1,2,\ldots,n$.
The $i$-th vector subspaces of complete flags ${\cal U}$ and ${\cal V}$ 
are denoted by $U_i$ and $V_i$, respectively.
Lemma~\ref{lem:asymptotic} rephrases as: 
$\lambda \cdot {\cal U} = \mu \cdot {\cal V}$ if and only 
if $\lambda = \mu$ and $U_i=V_i$ for each $i \in [n-1]$ with $\lambda_i > \lambda_{i+1}$.
By associating $\lambda \cdot {\cal U}$ 
with formal sum $\sum_{i=1}^n (\lambda_i - \lambda_{i+1}) U_i$ with $\lambda_{n+1} := 0$,  
the boundary $CP^{\infty}_{n}$ is also identified with the set of all formal sums 
\begin{equation}\label{eqn:formal_sum}
p= \sum_{i} \alpha_i U_i,
\end{equation}
where vector subspaces $U_i$ form a (partial) flag 
and nonzero coefficients $\alpha_i \in \RR$ are positive if $U_i \neq \CC^n$. 
Note that this expression (\ref{eqn:formal_sum}) is unique.
In this way, any single vector subspace $X (= 1 X)$ is viewed as 
a point of $CP_n^{\infty}$.

A  Weyl chamber $C \subseteq CP^{\infty}_n$
consists of $\lambda \cdot {\cal U}$ over all arranged vectors $\lambda$ 
for a (fixed) flag ${\cal U}$. 
The distance $d^{\infty}(p,q)$ of two points 
$\lambda \cdot {\cal U}$ and $\mu \cdot {\cal U}$ in the same chamber  
is given by $\|\lambda - \mu \|_2$.
Accordingly, the distance of any two points in $CP^{\infty}_n$ is given by 
the length metric of the two points (the infimum of the length of a path connecting them).  
The whole space $CP^{\infty}_n$ is a polyhedral cone complex
obtained by gluing these Euclidean polyhedral cones.
An apartment $E(g) := F(g)^{\infty} = \{g e^{\infty \diag \lambda}g^{\dagger} \mid \lambda \in \RR^n\}$ for $g \in GL(n,\CC)$ 
consists of all points of form $\lambda \cdot {\cal U}$
such that each subspace $U_i$ in flag ${\cal U}$ is spanned by column vectors of $g$.
$GL(n,\CC)$ acts isometrically on  $CP^{\infty}_n$ by 
$(g, \lambda \cdot {\cal U}) \mapsto \lambda \cdot g {\cal U}$.
\begin{figure}[t]
	\begin{center}
		\includegraphics[scale=0.6]{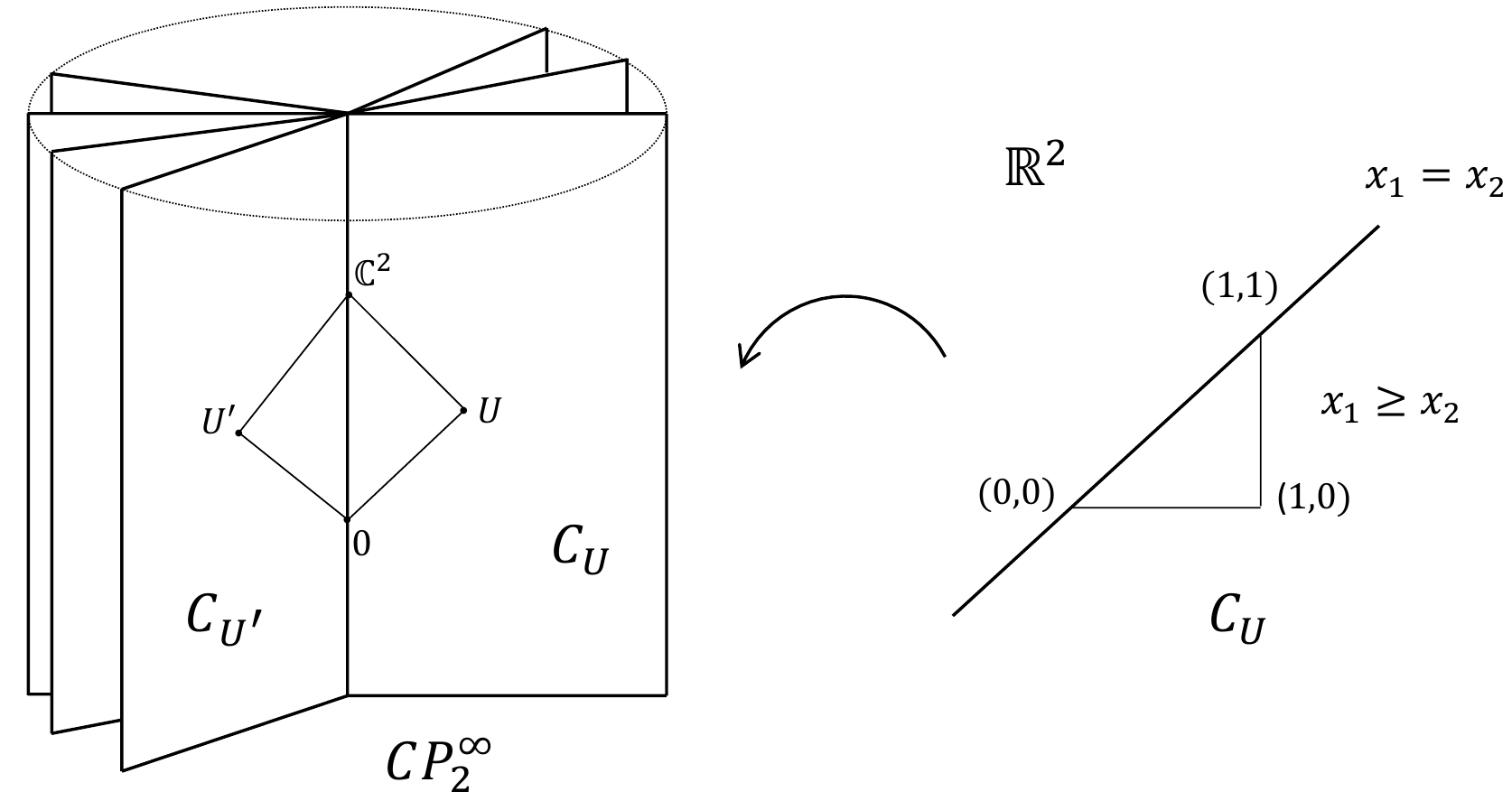}
		\caption{The boundary $CP_2^{\infty}$ of $P_2$}
		\label{fig:CP_2^infty}
	\end{center}
\end{figure}\noindent
\begin{Ex}
	Consider the case of $n=2$. Any complete flag ${\cal U}$ 
	is uniquely determined by its $1$-dimensional subspace $U=U_1$.
	The corresponding Weyl chamber 
	is a form of 
	$\{ \lambda \cdot {\cal U} \mid \lambda \in \RR^2, \lambda_1 \geq \lambda_2 \}$, 
	and is isometric to half-plane $\{ x \in \RR^2 \mid x_1 \geq x_2 \}:=C_{U} \subseteq \RR^2$.
	Then $CP_2^\infty$ is obtained by gluing 
    $C_{U}$ for all $1$-dimensional subspaces $U$, along the line of $x_1=x_2$.
    Specifically, 
    it is the disjoint union $\coprod_U C_U$ over all $1$-dimensional subspaces $U$ 
    modulo the equivalence relation:  $(U,x)  \sim (U',x')$ 
    if and only if $x_1=x_2=x_1'=x_2'$. Subspaces $U$ and $\CC^2$ 
    are the points of $CP_2^{\infty}$ that are the images of $(U, (1,0))$ and $(U, (1,1))$, respectively. See Figure~\ref{fig:CP_2^infty}. 
	This shape of $CP_2^{\infty}$ can be directly seen from the expression $CP_2^{\infty} = \RR  \times C(SL(2,\CC)/SU(2))^{\infty}$. Here $SL(2,\CC)/SU(2)$ is a $3$-dimensional hyperbolic space, and $C(SL(2,\CC)/SU(2))^{\infty}$ is an infinite star (Example~\ref{ex:hyperbolic}).
\end{Ex}

In this setting, we present explicit descriptions of Busemann functions, asymptotic gradients, and inner product $\langle, \rangle$ on $CP_{n}^{\infty}$.
For $g \in GL(n,\CC)$, let $[g]$ denote the image in $G/B$, 
which is the complete flag of vector subspaces spanned 
by the first $i$ columns of $g$ for $i=1,2,\ldots,n$.
For a Hermitian matrix $H$ and subset $I \subseteq [n] := \{1,2,\ldots,n\}$, 
let $H[I]$ denote the principal matrix of $H$ consisting of 
row/column indices in $I$,
where $H[\emptyset] :=1$.
\begin{Lem}[{see \cite[II.10.69]{BrHa}}]\label{lem:b_p_P_n}
	For $p = \lambda \cdot [u] \in CP_n^{\infty}$ with $u \in U(n)$, it holds
	\begin{eqnarray}
	b_p(x) &=& - \sum_{i=1}^n \lambda_i \log \frac{\det (u^{\dagger}x u)[\{i,\ldots,n\}]}{\det (u^{\dagger}x u)[\{i+1,\ldots,n\}]}, 
	\label{eqn:b_p_P_n} \\
	\nabla b_p(x) &=& - ub \diag \lambda b^\dagger u^\dagger, \label{eqn:nabla_b_p_P_n}
	\end{eqnarray}
	where $u^{\dagger}x^{1/2} = bk$ 
	for upper triangular matrix $b$ and unitary matrix $k$  (Gram–Schmidt orthonormalization).
\end{Lem}
\begin{proof}
	It suffices to consider $p \in P_n^{\infty}$. 
	The geodesic ray $c$ issuing from $I$ with $c(\infty) = p$ is written as 
	$c(t) = u e^{t \diag \lambda} u^{\dagger}$. 
	Therefore, we have $b_p(x) = \lim_{t\to \infty} d(u e^{t \diag \lambda}u^{\dagger}, x) - t = \lim_{t\to \infty} d(e^{t \diag \lambda}, u^{\dagger}xu) - t$.
	Suppose that all $\lambda_i$ are different. 
	Decompose $u^{\dagger}xu$ as $u^{\dagger}xu = n \diag r n^{\dagger}$, 
	where $n$ is an upper triangular matrix having $1$ on each diagonal 
	and  $r \in \RR^n$ is a positive vector with $r_i$ written as
	\begin{equation}\label{eqn:r_i}
	r_i = \frac{\det (u^{\dagger}xu) [\{i,\ldots,n\}]}{\det (u^{\dagger}xu) [\{i+1,\ldots,n\}]}. 
	\end{equation}
	As in the proof of Lemma~\ref{lem:asymptotic}, it holds
	$\lim_{t \to \infty} d(n^{-1}e^{t\diag \lambda} n^{-\dagger}, e^{t \diag \lambda}) = 0$, and we have $b_p(x) = \lim_{t\to \infty} d(e^{t \diag \lambda}, \diag r) - t  
	= \lim_{t\to \infty} \|t \diag \lambda - \log r \|_2 -t
	= - \sum_{i=1}^n \lambda_i \log r_i$ (see Example~\ref{ex:Euclidean_Busemann}), 
	where $\log r \in \RR^n$ is defined by $(\log r)_i := \log r_i$.
	Then we obtain (\ref{eqn:b_p_P_n}) from (\ref{eqn:r_i}). 
	If some of $\lambda_i$ are equal, we decompose $u^{\dagger} x u$ 
    to $n y n^{\dagger}$, 
    where  $n$ is an upper triangular matrix satisfying $n_{ii} =1$ and
	$n_{ij} = n_{ji} = 0$ if $\lambda_i = \lambda_j$ and $i \neq j$, and $y$ is a block diagonal matrix with $y_{ij} = y_{ji} = 0$ if $\lambda_i \neq \lambda_j$. 
	As above, it holds $b_p(x) = \lim_{t\to \infty} d(e^{t \diag \lambda}, y)-t$.
	Diagonalizing $y$ in each block by unitary matrices, 
	we obtain the same formula.
	
	Next we verify the second equation.
	For $H := ub \diag \lambda b^{\dagger} u^{\dagger}$, 
	consider the geodesic $t \mapsto x^{1/2} e^{t x^{-1/2} H x^{-1/2}}x^{1/2}$ (issuing from $x$).
	By $x^{1/2} = ubk = k^{\dagger}b^{\dagger} u^{\dagger}$ 
	and $x^{-1/2} = k^{\dagger}b^{-1}u^{\dagger} 
	= u b^{- \dagger} k$, we have 
  $x^{1/2} e^{t x^{-1/2} H x^{-1/2}}x^{1/2} = u b e^{t \diag \lambda} b^{\dagger} u^{\dagger}$. Hence, 
  by Lemma~\ref{lem:asymptotic}, 
  this geodesic is asymptotic to $c$, i.e., $t \mapsto u e^{t \diag \lambda} u^{\dagger}$.
  From Lemma~\ref{lem:nabla_Busemann},  we have $\nabla b_p(x) = - H$. 
\end{proof}

Let $\| \cdot \|_{\rm F}$ denote the Frobenius norm; 
then $\| H \|_x = \| x^{-1/2} H x^{-1/2}\|_{\rm F}$ for $H \in T_x$.

\begin{Prop}\label{prop:nabla^inf_P_n}
	Let $f:P_n \to \RR$ be a smooth convex function.
	\begin{itemize}
		\item[(1)] For $p = \lambda \cdot [u]$ 
	with unitary matrix $u$, it holds
	\begin{equation}\label{eqn:|nabla f+b_p|}
	\| \nabla (f + b_p)(x) \|_x = \|  b^{\dagger} u^{\dagger} df(x) u b - \diag \lambda  \|_{\rm F},
	\end{equation}
	where $u^{\dagger}x^{1/2} = bk$ 
	for upper triangular matrix $b$ and unitary matrix $k$.
	In particular, if $\nabla^{\infty}f(x) = p$, then $x^{1/2}df(x)x^{1/2} =  k^{\dagger} \diag \lambda k$.
	\item[(2)] Let $s$ denote the projection $\lambda \cdot {\cal U} \mapsto \lambda$. Then it holds
	\begin{equation}\label{eqn:sB(f^infty)}
	\overline{s\nabla^{\infty}f(P_n)} = \overline{s \dom f^*} = s B(f^{\infty}) = \bigcup_{C} s(C \cap B(f^{\infty})),
	\end{equation}
	where $C$ ranges over all Weyl chambers.
	\end{itemize}
\end{Prop}
We will see in Section~\ref{subsec:orbit} that $\overline{s\nabla^{\infty}f(P_n)}$
is viewed as an analogue of the moment polytope. 
\begin{proof} (1).
	From $\langle \nabla f(x), H \rangle_x  
	= \trace df(x) H$, we have $\nabla f(x) = x df(x) x$.
    Then $\| \nabla (f + b_p)(x) \|_x = \|x^{-1/2} (x df(x) x -  ub \diag \lambda b^\dagger u^\dagger )x^{-1/2}\|_{\rm F} = \|  k^{\dagger} (b^{\dagger} u^{\dagger} df(x) u b - \diag \lambda)k  \|_{\rm F} = \|  b^{\dagger} u^{\dagger} df(x) u b - \diag \lambda  \|_{\rm F}$. 
    
    (2). By (\ref{eqn:inclusion}), the inclusion $
    \overline{s\nabla^{\infty}f(P_n)} \subseteq \overline{s \dom f^*} \subseteq s B(f^{\infty})
    $ is clear. Take $p = \lambda \cdot [u] \in B(f^{\infty})$.
    By Theorem~\ref{thm:char_of_B(f^infty)}, there is a sequence $(x_i)$ in $P_n$ such that $\lim_{i \to \infty} \| \nabla (f+b_{p})(x_i)  \|_{x_i} = 0$. Suppose that $\nabla^{\infty}f(x_i) = \lambda_i \cdot [u_i]$.
    Via the decomposition $u^{\dagger} x_i^{1/2} = b_k k_i$ and 
    $u^{\dagger}_ix^{1/2}_i = b'_ih_i$, it holds  $x_i^{1/2}df(x)x_i^{1/2} =  h_i \diag \lambda_i h_i^{\dagger}$, and $s \nabla^{\infty}f(x_i) = \lambda_i$.
    By the above calculation,
    we have  $\| \nabla (f+b_{p})(x_i)  \|_{x_i} = \|h_i^{\dagger} \diag \lambda_i h_i - k_i^{\dagger} \diag \lambda k_i \|_{\rm F} \to 0$ $(i \to \infty)$. This implies that  $s \nabla^{\infty}f(x_i) = \lambda_i \to \lambda$, and $\lambda \in \overline{s \nabla^{\infty}f(P_n)}$.
\end{proof}

\begin{Lem}\label{lem:<>}
	For two points $p = \lambda \cdot {\cal U}$, $q = \mu \cdot {\cal V}$ in $CP^{\infty}_n$, it holds
	\begin{equation}\label{eqn:<>}
	\langle p,q \rangle = \sum_{1 \leq i,j \leq n} (\lambda_i - \lambda_{i+1}) 
	(\mu_{j}- \mu_{j+1}) \dim U_{i} \cap V_{j}.
	\end{equation}
\end{Lem}
\begin{proof}
	It is well-known (see \cite[II. 10.80]{BrHa}) that
	there are $g \in GL(n,\CC)$ and a permutation matrix $\sigma$ such that 
	${\cal U} = [g]$ and ${\cal V} = [g \sigma]$; this is nothing but 
	an axiom of building. 
	In particular, both $p=ge^{\infty \diag \lambda}g^{\dagger}$ 
	and $q = ge^{\infty \sigma \diag \mu \sigma^{\top}}g^{\dagger}$ belong to the apartment $F(g)^{\infty}$.  They are regarded as points in $\RR^n$:
    $p = \lambda = \sum_{i} (\lambda_i - \lambda_{i+1}) {\bf 1}_{[i]}$
	and $q =  \sum_{i} (\mu_i - \mu_{i+1}) {\bf 1}_{\sigma [i]}$, 
	where ${\bf 1}_J$  denotes the $n$-dimensional 0,1-vector 
	taking $1$ only on indices in $J \subseteq [n]$.
    Then $\langle p,q \rangle = \sum_{i,j} (\lambda_i - \lambda_{i+1}) 
    (\mu_{j}- \mu_{j+1}) |[i] \cap \sigma [j]|$.
    Here $|[i] \cap \sigma [j]| = \dim U_i \cap V_j$.
\end{proof}
Kapovich, Leeb, and Millson~\cite[Lemma 6.1]{KLM2009JDG} gives
the corresponding formula of the angle  of 
two vector subspaces regarded as points in $(SL(n,\CC)/SU(n))^{\infty}$.

\paragraph{Connections to submodular functions.}

Let ${\cal S}(\CC^n)$ denote the family of all vector subspaces of $\CC^n$.
A function $\rho: {\cal S}(\CC^n) \to \RR \cup \{\infty\}$ is called {\em submodular} 
if it satisfies
\begin{equation}\label{eqn:submodular}
\rho(X) + \rho(Y) \geq \rho(X \cap Y) + \rho(X + Y) \quad (X,Y \in {\cal S}(\CC^n)).
\end{equation}
This extends the classical submodular functions, 
which are functions $\kappa$ on $2^{[n]}$ satisfying $\kappa(X)+ \kappa(Y) \geq \kappa(X \cap Y) + \kappa(X \cup Y)$; see e.g., \cite{FujiBook,MurotaDCA,Schrijver}. 
A submodular function $\rho$ with $\rho(\{0\}) = 0$ gives rise to a positively homogeneous function 
$\overline{\rho}: CP_n^{\infty} \to \RR \cup \{\infty\}$ via piecewise linear extension 
\begin{equation}
\overline{\rho}(p) := \sum_{i=1}^n (\lambda_i- \lambda_{i+1}) \rho(U_i) \quad (p= \lambda \cdot {\cal U} \in CP_n^{\infty}).
\end{equation} 
This is an analogy of the {\em Lov\'asz extension}~\cite{Lovasz83} in the classical setting.
So we call $\overline{\rho}$ the Lov\'asz extension of $\rho$.
This is equivalent to the one considered in \cite{HamadaHirai,HH16L-convex}, where $\lambda$ is restricted to $\lambda \in  [0,1]^n$.
\begin{Prop}[{\cite[Theorem 3.9]{HH16L-convex}}]
	For a positively homogeneous function $h: CP_n^{\infty} \to \RR\cup \{\infty\}$, the following are equivalent:
	\begin{itemize}
		\item[(i)] $h$ is a convex function that is affine on each Weyl chamber.
		\item[(ii)] $h$ is the Lov\'asz extension of a submodular function on ${\cal S}(\CC^n)$. 
	\end{itemize} 
\end{Prop}
The proof reduces to the classical convexity characterization~\cite{Lovasz83} by restricting $h$ to each apartment.
\begin{Lem}\label{lem:submodular_<>}
		Let $q = \mu \cdot {\cal V} \in CP_n^{\infty}$. Then $p \mapsto - \langle q, p \rangle$ is the Lov\'asz extension of submodular function 
		\begin{equation}\label{eqn:submo_dim_U^X}
		X \mapsto - \langle q,X\rangle = - \sum_{i=1}^n (\mu_i- \mu_{i+1}) \dim V_i \cap X.
		\end{equation}
\end{Lem}
\begin{proof}
	By Lemma~\ref{lem:b^infty}, $p \mapsto - \langle q, p \rangle$ is convex.
	Also it is an affine function on any apartment containing $q$.
	Necessarily, it is affine on every Weyl chamber. 
	Therefore, if suffices to show that (\ref{eqn:submo_dim_U^X}) is submodular. 
	We first show that $X \mapsto - \dim V \cap X$ is submodular.
	This follows from $\dim V \cap X + \dim V \cap Y = \dim V \cap X \cap Y+ \dim ((V \cap X)+ (V \cap Y))$ and $V \cap X + V \cap Y \subseteq V \cap (X + Y)$, 
	we have submodularity of each summand in (\ref{eqn:submo_dim_U^X}) with $i\neq n$.
	For $i=n$ $(V_n=\CC)$, 
	the equality $\dim X \cap V_n + \dim X \cap V_n = \dim (X \cap Y) \cap V_n + \dim (X \cup Y) \cap V_n$ 
	holds, and 
	hence (\ref{eqn:submo_dim_U^X}) is submodular (with taking zero on $\{0\}$). 
\end{proof}

For a submodular function $\rho: {\cal S}(\CC^n) \to \RR \cup \{\infty\}$ with $\rho(\{0\}) =0$, 
the subset $B(\overline{\rho})$ is described by fewer inequalities indexed by vector subspaces: It equals 
the set of points $p \in CP_n^{\infty}$ satisfying
\begin{equation}\label{eqn:base_polyhedron}
\langle X, p \rangle \leq \rho(X) \  (X \in {\cal S}(\CC^n)), \quad  \langle \CC^n,p \rangle = \rho(\CC^n).
\end{equation}
Indeed, if $p$ satisfies (\ref{eqn:base_polyhedron}), then for 
$q = \mu \cdot {\cal V} \in CP_n^{\infty}$, 
it holds $\langle q, p \rangle = \sum_{i} (\mu_i- \mu_{i+1}) \langle V_i,p \rangle 
\leq \sum_{i} (\mu_i- \mu_{i+1}) h(V_i) = h(q)$.  

Then $B(\overline{\rho})$ is called the {\em base polyhedron} of $\rho$, 
which is clearly an analogue of the classical one.
A convex function $f$ on $P_n$ is said to 
be {\em asymptotically submodular}
if $f^{\infty}$ is the Lov\'asz extension of a submodular function.
In this case, the condition (c) in Theorem~\ref{thm:main} 
can be replaced by (\ref{eqn:base_polyhedron}) with $\rho = f^{\infty}$.
Accordingly, the convex optimization problem (\ref{eqn:optimization}) 
becomes ``discrete" convex optimization (submodular function minimization) over the lattice of vector subspaces: 
\begin{equation}
\mbox{inf.} \quad f^{\infty}(X) - \langle p, X \rangle \quad \mbox{s.t.} \quad X \in {\cal S}(\CC^n).
\end{equation}

\section{Scaling problems}\label{sec:scaling}

In this section, we explain how the results in the previous sections 
are applied to operator scaling and its generalizations.
In our argument, the following convex function on $\RR^n$
plays important roles. This function appears in proving (semi)stability results 
(Kempf-Ness theorem, Hilbert-Mumford criterion) 
in invariant theory; see \cite{KempfNess,Wallach_GIT}.
\begin{Lem}\label{lem:log_sum_exp}
	For $a_i > 0$ and $w_i \in \RR^n$ ($i=1,2,\ldots,m$), 
	define $f: \RR^n \to \RR$ by
	\begin{equation}
	f(x) := \log \sum_{i=1}^m a_i e^{\langle w_i, x \rangle}\quad (x \in \RR^n).
	\end{equation}
	Then $f$ is convex, where: 
	\begin{eqnarray*}
	&&	f^{\infty}(p) = \max_{i=1,2,\ldots,m} \langle w_i, p \rangle \quad (p \in \RR^n),\\
	&& \dom f^* = B(f^{\infty}) = \mbox{the convex hull of $w_1,w_2,\ldots,w_m$}.
	\end{eqnarray*}
\end{Lem}	
	\begin{proof}
		The convexity of $f$ is well-known. Computing $f^{\infty}$ 
		is also a little exercise \cite[p. 68]{Rockafellar}:
	Letting $v^*:= \max_{i} \langle w_i, p \rangle$, we have 
	$f^{\infty}(p) = 
	\lim_{t \to \infty} \frac{1}{t} \log \sum_{i=1}^m a_i e^{\langle w_i, p t \rangle} 
	= v^* + \lim_{t \to \infty} \frac{1}{t} \log \sum_{i=1}^m a_i e^{t (\langle w_i, p \rangle- v^*)} = v^*$.
	
	Thus $B(f^{\infty})$ equals the convex hull of $w_i$. We show that 
	it equals $\dom f^*$. Suppose $p = \sum_{i} \lambda_i w_i \in B(f^{\infty})$ with
	$\lambda_i \geq 0$ and $\sum_i \lambda_i = 1$.
	Then, for small $c > 0$, it holds $c \lambda_i \leq a_i$ for all $i$. 
	From $\log \sum_{i=1}^m a_i e^{\langle w_i, x \rangle} \geq  \log \sum_{i=1}^m c \lambda_i e^{\langle w_i, x \rangle} \geq \log c +  \log e^{\sum_{i} \lambda_i \langle w_i, x \rangle} = \log c + \langle p,x \rangle$, we have $f(x) - \langle p,x \rangle \geq \log c$ for all $x \in \RR^n$. Hence $B(f^{\infty}) \subseteq \dom f^*$; the reverse inclusion  is generally true.
\end{proof}

\subsection{Operator scaling with specified marginals}

Let $A = (A_1,A_2,\ldots,A_m)$ be an $m$ tuple of nonzero complex $n \times n$ matrices.  
Let $\lambda,\mu \in \RR^n$ be nonnegative arranged vectors with 
the same sum $\sum_{i}\lambda_i = \sum_{i}\mu_i = n$ (say).
The {\em operator scaling problem with marginal $\lambda,\mu$}, introduced by Franks~\cite{Franks18}, is 
to find a pair of nonsingular matrices $g,h \in GL(n, \CC)$ such that
\begin{equation}\label{eqn:scaling}
\sum_{k=1}^m g^{\dagger}A_k h h^{\dagger} A_k^\dagger g = \diag \lambda, \quad 
\sum_{k=1}^m h^{\dagger}A_k^{\dagger} gg^{\dagger} A_k h = \diag \mu.
\end{equation}
If such $g,h$ exist, then $A$ is said to be {\em $(\lambda, \mu)$-scalable}.
If for every $\epsilon > 0$ there are $g,h$ such that
\begin{equation}\label{eqn:approximatescalability}
\left\| \sum_{k=1}^m g^{\dagger}A_k h h^{\dagger} A_k^\dagger g - \diag \lambda \right\|_{\rm F} < \epsilon, \quad \left\| \sum_{k=1}^m h^{\dagger}A_k^{\dagger} gg^{\dagger} A_k h - \diag \mu \right\|_{\rm F} < \epsilon,
\end{equation}
then $A$ said to be {\em approximately $(\lambda, \mu)$-scalable}.
If $\lambda =\mu = {\bf 1}$, it is
the original operator scaling problem by Gurvits~\cite{Gurvits04}.
	In this case, the approximate scalability is equivalent to 
the {\em noncommutative nonsingularity} of symbolic matrix $\sum_{k}A_k x_k$ \cite{FortinReutenauer04,IQS15a}.  
See \cite{Franks18,FranksSomaGoemans2022,GGOW} for further applications of operator scaling.
 
For simplicity, we assume that at least one of 
$\bigcap_{k=1}^m \ker A_k$ and $\bigcap_{k=1}^m \ker A_k^{\dagger}$
is trivial $\{0\}$. 
Otherwise, by coordinate change, we can make $A$ 
satisfy $(A_{k})_{in} = (A_{k})_{nj} = 0$ for $i,j,k$.
Then the problem reduces 
to the upper left $(n-1) \times (n-1)$ submatrices.

The operator scaling problem is viewed as the problem 
of finding a point $(x,y)$ in $P_n \times P_n$ 
at which the following convex function $f_{A}: P_n \times P_n \to \RR$ has a specified asymptotic gradient:
\begin{equation}
f_{A}(x,y) := n \log \sum_{k=1}^m \trace  x A_k y A_k^{\dagger}.
\end{equation}
This function is known to be (geodesically) convex. 
\begin{Lem}[{See e.g., \cite{AGLOW}}]
	$f_A$ is convex.	
\end{Lem}
Indeed, 
	on a maximal flat $F = F(g) \times F(h) = \{g e^{\diag \alpha} g^{\dagger} \}_{\alpha \in \RR^n} \times \{h e^{\diag \beta} h^{\dagger} \}_{\beta \in \RR^n}$, $f_{A}$ is written as
	\begin{equation}\label{eqn:on_flat}
		f_{A}(ge^{\diag \alpha}g^{\dagger},h e^{\diag \beta} h^{\dagger}) 
		=  n \log \sum_{1 \leq i,j \leq n} a_{ij}(g,h) e^{\alpha_i+\beta_j} \quad 
		(\alpha, \beta \in \RR^n),
	\end{equation}
	where $a_{ij}(g,h) := \sum_{k=1}^m |(g^{\dagger} A_k h)_{ij}|^2$.
	By Lemma~\ref{lem:log_sum_exp},
	$f_{A}$ is convex in every flat. 

In addition to the $(\lambda, \mu)$-scalability, 
we consider a sharper scalability concept.
Let ${\cal U}, {\cal V}$ be complete flags, 
and consider points $(\lambda \cdot {\cal U}, \lambda \cdot {\cal V})$ in the boundary 
$C(P_n \times P_n)^{\infty} = CP_n^{\infty} \times CP_n^{\infty}$. 
We say that $A$ is {\em $(\lambda \cdot {\cal U}, \mu \cdot {\cal V})$-scalable} 
if there are $g,h \in GL(n, \CC)$ such that 
$([g],[h]) = ({\cal U}, {\cal V})$ and (\ref{eqn:scaling}) hold.
Accordingly, 
we say that $A$ is {\em approximately $(\lambda \cdot {\cal U}, \mu \cdot {\cal V})$-scalable}
if for every $\epsilon > 0$ there are $g,h \in GL(n,\CC)$ 
such that $([g],[h]) = ({\cal U}, {\cal V})$ and (\ref{eqn:approximatescalability}) hold.
By definition, 
$A$ is (approximately) $(\lambda, \mu)$-scalable if
and only if $A$ is (approximately) $(\lambda \cdot {\cal U}, \mu \cdot {\cal V})$-scalable
for some flags ${\cal U},{\cal V}$.

When ${\cal U}$ and ${\cal V}$ are standard flag ${\cal E} := [I]$,
scaling matrices $g,h$ are upper triangular, and
hence the (approximate) $(\lambda \cdot {\cal E}, \mu \cdot {\cal E})$-scalability is equivalent to (approximate) $(\lambda, \mu)$-scalability by 
triangular matrices in the sense of Franks~\cite{Franks18}.
Note that the $(\lambda \cdot {\cal U}, \mu \cdot {\cal V})$-scalability reduces to 
the triangular scalability, since  
$A$ is $(\lambda \cdot {\cal U}, \mu \cdot {\cal V})$-scalable if 
and only if $g^{\dagger} A h$ is $(\lambda \cdot {\cal E}, \mu \cdot {\cal E})$-scalable 
for $([g],[h]) = ({\cal U}, {\cal V})$.

The $(\lambda \cdot {\cal U}, \mu \cdot {\cal V})$-scalability is rephrased by
using asymptotic gradient $\nabla^{\infty}$ and Busemann functions.
\begin{Prop}\label{prop:scalability}
	\begin{itemize}
		\item[(1)]  $A$ is $(\lambda \cdot {\cal U},\mu \cdot {\cal V})$-scalable
		if and only if there are points $x,y$ in $P_n$ such that $\nabla^{\infty}f_A(x,y) = (\lambda \cdot {\cal U},\mu \cdot {\cal V})$.
		\item[(2)]   $A$ is approximately $(\lambda \cdot {\cal U},\mu \cdot {\cal V})$-scalable
		if and only if  $\inf_{x,y \in P_n} \| \nabla (f_A + b_{\lambda \cdot {\cal U}, \mu \cdot {\cal V}})(x,y)\|_{x,y} = 0$.
	\end{itemize}
\end{Prop}
\begin{proof}
	From $df_{A}(x,y)(H,G) =\frac{\rm d}{{\rm d}t}\mid_{t=0} n \log \sum_{k=1}^m \trace  (x+tH) A_k (y+tG) A_k^{\dagger}$, 
		we have 
	\begin{equation}\label{eqn:df_A}
		d f_{A}(x,y) = C_{x,y} \left( \sum_{k} A_k y A_k^{\dagger}, \sum_{k} A_k^{\dagger} x A_k \right),
	\end{equation}
	where $C_{x,y} := n/\sum_{k=1}^m \trace  x A_k y A_k^{\dagger}$.
	Let ${\cal U} =[u]$ and ${\cal V} = [v]$ for $u,v \in U(n)$.
	By Proposition~\ref{prop:nabla^inf_P_n}, we have
	\begin{eqnarray}
	\|\nabla (f_A+b_{\lambda \cdot {\cal U}, \mu \cdot {\cal V}})(x,y)\|_{x,y}^2 &=&
	\left\|C_{x,y}  \sum_{k} b^{\dagger} u^{\dagger} A_k v c c^{\dagger}v^{\dagger} A_k^{\dagger} u b  - \diag \lambda \right\|_{\rm F}^2 \nonumber \\ 
	&& +  \left\| C_{x,y}  \sum_{k}c^{\dagger} v^{\dagger} A_k^{\dagger} u b b^{\dagger}u^{\dagger} A_k v c  - \diag \mu  \right\|_{\rm F}^2, \label{eqn:|nabla_f_A+b|}
	\end{eqnarray}
	where $u^{\dagger} x^{\frac{1}{2}} = b k$ and $v^{\dagger} y^{\frac{1}{2}} = ck'$ for $k,k' \in U(n)$ and upper-triangular matrices $b,c$. From this, we have the claims, where 
   required scaling matrices $g,h$ are given as $g=  C_{x,y}^{1/4} ub$ and $h = C_{x,y}^{1/4} vc$ with $[g] = [u] = {\cal U}$ and $[h] = [v] ={\cal V}$. 
\end{proof}

For the function $f_A$, the inclusion $\dom f_A^* \subseteq \overline{\dom f_A^*} \subseteq B(f_A^{\infty})$ 
becomes equality.
\begin{Prop}\label{prop:domf_A*}
	$\dom f_A^* = B(f_A^{\infty})$.
\end{Prop}
We will prove a general version (Proposition~\ref{prop:generalversion}) in Section~\ref{subsec:orbit}. 
Thus, the $(p,q)$-scalability with
 $p = \lambda \cdot {\cal U}$ and $q = \mu \cdot {\cal V}$ 
can be decided by the boundeness of convex optimization: 
\begin{equation*}
{\rm inf.} \ (f_{A} + b_{p,q}) (x,y) = f_{A}(x,y) + b_p (x) + b_q(y) \ {\rm s.t.} \ (x,y) \in P_n \times P_n,
\end{equation*}
where Busemann functions $b_p$ and $b_q$ are explicitly given by Lemma~\ref{lem:b_p_P_n}.
By optimizing $y$ under a fixed $x \in P_n$, we may minimize function 
$g_{A,q} :P_n \to \RR$:
\begin{equation}\label{eqn:g_A,p}
	g_{A,q}(x) := \inf_{y \in P_n} f_{A}(x,y) + b_{q}(y) \quad (x \in P_n).
\end{equation}
One can see from (\ref{eqn:df_A}) and (\ref{eqn:|nabla_f_A+b|}) that optimal $y$ is obtained by $y = hh^{\dagger}$ for $h \in GL(n,\CC)$ with ${\cal V} = [h]$ and $h^{\dagger}(\sum_{k}A_kxA_k^{\dagger})h = \diag \mu$.
When ${\cal U} ={\cal V} ={\cal E}$, 
the infimum of $f_{A}+ b_p + b_q$ (or $g_{A,q} + b_p$) equals
(up to constant) the logarithm of the {\em capacity of specified marginal} in Franks~\cite{Franks18}.\footnote{
	To see the consistency with his formulation, 
	use the relation 
	$b_{\lambda \cdot {\cal E}}(gg^{\dagger}) = - \log \det (\diag \lambda, g^{\dagger}g)$
	for any upper-triangular matrix $g$, where 
	$\det (\diag \lambda, g^{\dagger}g)$ is the {\em relative determinant} 
	in the sense of~\cite{Franks18}.}

%

We compute explicit descriptions of the recession functions of $f_{A}$
and the associated subset $B(f^{\infty}_{A})$. 
See Remark~\ref{rem:123}~(1) for $g_{A,q}$ and  $B(g^{\infty}_{A})$.
Let ${\cal S}_{A}$ be the family of all pairs $(X,Y)$ of vector subspaces in $\CC^n$ such that $u^{\dagger} A_k v = 0$ for all $u \in X, v \in Y, k \in [m]$.
\begin{Prop}\label{prop:f_Ainfty}
	\begin{itemize}
		\item[(1)] The recession function $f_A^{\infty}$ is given by
		\begin{equation}\label{eqn:f_A^infty}
		f_A^{\infty}(p,q) = n \max \{\alpha_i + \beta_j \mid i,j \in [n]: (g^{\dagger}A_kh)_{ij} \neq 0\  (\exists k \in [m]) \},
		\end{equation}
		where $p = \alpha \cdot [g], q = \beta \cdot [h] \in CP^{\infty}_n$.
		 \item[(2)] 
		 $B(f^{\infty}_{A})$ is the set of $(p,q)\in CP_n^{\infty} \times  CP_n^{\infty}$ satisfying
		 \begin{equation}\label{eqn:X,Y}
		 \langle X,p \rangle + \langle Y,q\rangle \leq n \quad ((X,Y) \in {\cal S}_A).
		 \end{equation}
	\end{itemize}
\end{Prop}
\begin{proof}
	(1) follows from Lemma~\ref{lem:log_sum_exp} and the expression (\ref{eqn:on_flat}).
	(2). Let $(p,q) \in CP_n^{\infty} \times CP_n^{\infty}$.
	For $(X,Y) \in {\cal S}_A$, 
    choose $g,h \in GL(n,\CC)$ such that $g$ and $h$ span $X$ and $Y$ in the first $k$ and $l$ column subsets, respectively. 
    Then, each $g^{\dagger}A_kh$ has a $k \times l$ zero block in the upper left corner. 
    If $X,Y$ are viewed as points in $CP_n^{\infty}$ by (\ref{eqn:formal_sum}), then 
    $X = {\bf 1}_{[k]} \cdot [g]$ and $Y = {\bf 1}_{[l]} \cdot [h]$.
    By the assumption that
    $\bigcap_{k=1}^m \ker A_k=\{0\}$ or $\bigcap_{k=1}^m \ker A_k^{\dagger} =\{0\}$, the maximum in (\ref{eqn:f_A^infty}) 
    is attained by $i \in [k], j \in [n] \setminus [l]$ or $i \in [n] \setminus [k], j \in [l]$, 
    and  we have $f_A^{\infty}(X,Y) = n$.
    Thus (\ref{eqn:X,Y}) is  
	a necessary condition for $(p,q) \in B(f_A^{\infty})$.
	Consider an apartment  $E(g) \times E(h) = F(g)^{\infty} \times F(h)^{\infty}$ containing $(p,q)$ and identify it with $\RR^n \times \RR^n$ 
	by $(ge^{\infty \alpha}g^{\dagger}, he^{\infty \beta}h^{\dagger}) \mapsto (\alpha, \beta)$.
	From Lemma~\ref{lem:log_sum_exp} and (\ref{eqn:on_flat}), 
	we have
	\begin{equation}
	B_{E(g) \times E(h)}(f^{\infty}) = n\, \mbox{the convex hull of $e_i + f_j$ 
		for all $i,j$ with $a_{ij}(g,h) \neq 0$},
	\end{equation}
	where $e_i$ and $f_i$ denote the $i$-th unit vectors 
	of $E(g)$ and $E(h)$, respectively.
	This is $n$ times
	the {\em clique polytope} of the bipartite graph $G$ 
	with vertex set $[n] \sqcup [n]$ and edge set $\{ ij \mid a_{ij}(g,h) \neq 0\}$; 
	see \cite[Section 65.4]{Schrijver}. 
	By a standard network flow argument, 
	we obtain the inequality description 
	of $B_{E(g) \times E(h)}$ as
		\begin{equation}\label{eqn:S,T}
		\sum_{i \in S} \alpha_i  + \sum_{j \in T} \beta_j \leq n \quad  ((S,T) \subseteq [n] \times [n]: a_{ij}(g,h) = 0\ (i \in S,j \in T)),
		\end{equation}
		where $S \sqcup T$ is nothing but a stable set of the graph $G$.
	Notice that (\ref{eqn:S,T}) is the subsystem for (\ref{eqn:X,Y}) such that
	vector subspace $X$ and $Y$ are 
	spanned by columns vectors of $g$ and $h$.  
	By Proposition~\ref{prop:C_cap_B(h)}, 
	satisfying all such inequalities is also sufficient 
	for $(p,q) \in B(f_A^{\infty})$. 
\end{proof}
Thus, by Theorem~\ref{thm:char_of_B(f^infty)}, 
Propositions~\ref{prop:scalability}, \ref{prop:domf_A*}, \ref{prop:f_Ainfty}, and Lemma~\ref{lem:<>}, we have: 
\begin{Thm}[\cite{Franks18}]\label{thm:Franks}
	The following conditions are equivalent:
	\begin{itemize}
		\item[(a)] $A$ is approximately $(\lambda \cdot {\cal U}, \mu \cdot {\cal V})$-scalable. 
		\item[(b)] $\inf_{x,y \in P_n} f_{A}(x,y) + b_{\lambda \cdot {\cal U}}(x) + b_{\mu \cdot {\cal V}}(y) > -\infty$.
		\item[(c)] For all $(X,Y) \in {\cal S}_A$, it holds
		\begin{equation}\label{eqn:Franks}
		\sum_{i=1}^n (\lambda_i- \lambda_{i+1}) \dim U_i \cap X + 
		\sum_{i=1}^n (\mu_i- \mu_{i+1}) \dim V_i \cap Y \leq n.
		\end{equation}
	\end{itemize}
\end{Thm}
Franks~\cite{Franks18} showed that 
the approximate scalability
reduces to the triangular scalability in the generic case.
\begin{Thm}[\cite{Franks18}]
	$A$ is approximately $(\lambda, \mu)$-scalable if and only 
	if $g^{\dagger}Ah$ is approximately $(\lambda \cdot {\cal E}, \mu \cdot {\cal E})$-scalable
	for generic $g,h \in GL(n,\CC)$. 
\end{Thm}
Here ``generic" means that 
there is an affine variety $V \subseteq GL(n,\CC)^2$ such that
the latter property holds for all $(g,h) \in GL(n,\CC)^2 \setminus V$.
We will verify this theorem for a general setting of the moment polytope membership in the next section.
\begin{Rem}\label{rem:123}
	\begin{itemize}
		\item[(1)] 
		 One can show that the recession function $g_{A,q}^{\infty}$ of $g_{A,q}$ is the Lov\'asz extension of submodular function
		\begin{equation}\label{eqn:g_A,p^infty}
			X \mapsto n - \langle X^{\bot_{A}}, q \rangle,   
		\end{equation}
	where $X^{\bot_{A}}$ 
	denotes the maximum subspace $Y$ with $(X,Y) \in {\cal S}_{A}$.
		In particular, $g_{A,q}$ is asymptotically submodular,
		and $B(g_{A,q}^{\infty})$ coincides with the base polyhedron of $g_{A,q}^{\infty}$.
		\item[(2)]
Computation of the constant ({\em BL-constant}) 
of the {\em Brascamp-Lieb inequality}~\cite{BrascampLieb,Lieb}
is formulated as the same type of convex optimization 
over the product of PSD-cones (over $\RR$)~\cite{GGOW_GAFA}. 
The objective function is also asymptotically submodular. 
A finiteness characterization of the BL-constant by \cite{BCCT} can be deduced by 
the same way as for (\ref{eqn:Franks}) above.  
\item[(3)]
	Since $(X,Y),(X',Y') \in {\cal S}_A$ implies $(X\cap X',Y+Y'),(X+X',Y \cap Y') \in {\cal S}_A$, the function $(X,Y) \mapsto - \langle X,p \rangle - \langle Y,q \rangle$
	also admits a submodular function structure 
	on the lattice ${\cal S}_A$ with $\wedge = (\cap,+), \vee = (+,\cap)$.
	Its Lov\'asz extension (in the sense of \cite{HamadaHirai,HH16L-convex})
	coincides with a part of $f^{\infty}_{A,p,q}$.
	The nc-rank computation algorithm in~\cite{HamadaHirai} 
	is interpreted as minimizing $f^{\infty}_{A,p,q}$ (with $\mu= \lambda = {\bf 1}$) 
	over a convex neighborhood of $0$.
	\end{itemize}
\end{Rem}

\subsection{Optimization on group orbits}\label{subsec:orbit}
The operator scaling and its generalizations (e.g., tensor scaling~\cite{BGOWW_tensor0,BFGOWW_tensor})
can be formulated as optimization over an orbit of a group action.
We finally consider the generalized scaling problems formulated by B\"urgisser, Franks, Garg, Oliveira, Walter, and Wigderson~\cite{BFGOWW}. 
Let $G \subseteq GL(n,\CC)$ be a reductive algebraic group over $\CC$, 
i.e., $G$ is defined by the zero set of a finite number of polynomials with complex coefficients, and $g \in G$ implies $g^{\dagger} \in G$. 
We assume that $G$ is connected.
Since $G$ is a closed subgroup of Lie group $GL(n,\CC)$, it is also a Lie group.
Let $K := G \cap U(n)$ be a maximal compact subgroup of $G$.
Let $\mathfrak{g}$ and $\mathfrak{u}$ denote the Lie algebras of $G$ and $K$, 
respectively, where $\mathfrak{g} = \mathfrak{u}+ i \mathfrak{u}$ 
is the complexification of $\mathfrak{u}$ (or Cartan decomposition of involution $X \mapsto -(X)^{\dagger}$).
This is a situation of~\cite[VII. 2. Example (2)]{KanppLieGroup}.

Let $\pi: G \to GL(N,\CC)$ be a rational representation, i.e., 
each entry of matrix $\pi(g)$ is a polynomial of $g_{ij}$ and $(\det g)^{-1}$.
Let $\langle \cdot, \cdot \rangle$ denote a $K$-invariant inner product on $\CC^N$, 
i.e., it satisfies $\langle \pi(k)u, \pi(k)v \rangle = \langle u, v \rangle$ for all $k \in K$.
Let $\Pi:= d\pi(I)$ be the Lie algebra representation of $\pi$.
Then $\pi(e^{H}) = e^{\Pi(H)}$ holds for $H \in \mathfrak{g}$.
The conjugate of $g \in GL(N,\CC)$ with respect to $\langle \cdot, \cdot \rangle$
is denoted by $g^{\dagger}$ (the matrix satisfying
$\langle g^{\dagger}u,v\rangle= \langle u,gv\rangle$ for all $u,v \in \CC^N$).
Then $\Pi(H^{\dagger}) = \Pi(H)^{\dagger}$ holds for $H \in \mathfrak{g}$.\footnote{From $\frac{\rm d}{{\rm d}t} \mid_{t=0} \langle \pi(e^{tH_0})u, \pi(e^{tH_0})v\rangle =0$ for $H_0 \in \mathfrak{u}$, we have $\langle \Pi(H_0)u,v\rangle + \langle u, \Pi(H_0)v\rangle =0$.
Thus $\Pi(H_0)^{\dagger} = - \Pi(H_0)$. For $H = H_0+i H_1$ with $H_0,H_1 \in \mathfrak{u}$, 
we have $\Pi(H)^\dagger = \Pi(H_0)^{\dagger} - i \Pi(H_1)^{\dagger} = - \Pi(H_0) + i \Pi(H_1) = \Pi(-H_0+ i H_1) = \Pi(H^{\dagger})$.}
Also $\pi(g)^{\dagger} = \pi(g^{\dagger})$ holds for $g \in G$.\footnote{By $\pi(k^{\dagger}) = \pi(k^{-1}) = \pi(k)^{-1} = \pi(k)^\dagger$ for $k \in K$ 
and polar decomposition $g= k e^{iH}$ for $k \in K$ and $H \in \mathfrak{u}$, we have
$\pi(g^{\dagger}) = e^{- i \Pi(H^{\dagger})} \pi(k^{\dagger}) = e^{-i\Pi(H)^{\dagger}} \pi(k)^{\dagger} 
= (\pi(k)e^{i\Pi(H)})^{\dagger} = \pi(g)^{\dagger}$.}

Given a vector $v \in \CC^N$, 
consider minimization of log-norm $\log \|\pi(g)v\|^2$ (twice of the {\em Kempf-Ness function} in \cite{BFGOWW}) over the $G$-orbit of $v$:
\begin{equation}\label{eqn:rho(G)v}
\mbox{inf.} \ \log \|\pi(g)v\|^2 \quad \mbox{s.t.}\ g \in G.
\end{equation}
This optimization can decide whether 
$0 \in \overline {\pi(G)v}$ (the closure of orbit $\pi(G)v$), via the unboundedness. This is equivalent to the membership of $v$ in 
the {\em null-cone} of the invariant ring of $\pi$.
The operator scaling in the previous section corresponds to the {\em left-right action} 
$(g,h) \mapsto g^{\dagger} A h$,  
where $f_A$ is a constant multiple of the Kempf-Ness function.

Since the norm is $K$-invariant, the optimization problem (\ref{eqn:rho(G)v})
is viewed as that the quotient space $G/K$, 
which turns out to be a symmetric space of nonpositive curvature.
We formulate the optimization problem (\ref{eqn:rho(G)v}) more explicitly, 
as in \cite[Remark 3.4]{BFGOWW}.
Note that $\|\pi(g)v\|^2  = \langle v, \pi(g^{\dagger}g) v \rangle$, 
where $g^{\dagger}g \in G \cap P_n$. 
Since $G$ is algebraic, $x \in G \cap P_n$ implies $x^{1/2} \in G \cap P_n$; see \cite[II.10.59]{BrHa}. 
Therefore (\ref{eqn:rho(G)v}) is also written as
\begin{equation}
\mbox{inf. } f_v(x) := \log \langle v, \pi(x)v \rangle \quad
{\rm s.t.}\quad  x \in M := G \cap P_n.
\end{equation}
Here $M = G \cap P_n$ is a totally geodesic subspace of $P_n$, 
and hence is a symmetric space of nonpositive curvature 
(see \cite[II. 10. 50]{BrHa}). By polar decomposition $G = K e^{i\mathfrak{u}}$, 
we have $M = e^{i \mathfrak{u}} \simeq G/K$, where
$i\mathfrak{u}$ 
is viewed as the tangent space at $I$ with inner product $(\lambda,\nu) \mapsto \trace \lambda \nu$.

Then, $M$ is decomposed as a Euclidean space and symmetric space of noncompact type as follows.
Since $G$ is reductive, the Lie algebra $\mathfrak{g}$ is 
the direct sum of the center $\mathfrak{z}$ and semisimple Lie algebra $\mathfrak{g}_1 :=[\mathfrak{g},\mathfrak{g}]$, where $\mathfrak{z}$ and $\mathfrak{g}_1$ 
are orthogonal in the inner product $X,Y \mapsto {\rm Re} \trace (XY^{\dagger})$; see~\cite[Proposition 1.59]{KanppLieGroup}.
Now $G$ is commuting product $ZG_1$ of the center $Z$ for  $\mathfrak{z}$ and semisimple Lie group $G_1$ for $\mathfrak{g}_1$, where
$i\mathfrak{u} =\mathfrak{z} \cap i \mathfrak{u} +\mathfrak{g}_1 \cap  i \mathfrak{u}$; see \cite[Proposition 7.19 (e)]{KanppLieGroup}. 
Thus $M = G \cap P_n$
is Riemannian product of Euclidean space $\RR^k \simeq e^{\mathfrak{z} \cap i \mathfrak{u}}$ and symmetric space 
$M_1 := G_1 \cap P_n = e^{\mathfrak{g}_1 \cap i \mathfrak{u}}$ of noncompact type. 
If $g = zg_1$ for $z \in Z$ and $g_1 \in G_1$, 
then the action of $g$ on $M = \RR^k \times M_1$ is given   
so that $g_1$ acts on $M_1$ as $x \mapsto g_1 x g_1^\dagger$  (as before)
and $z$ acts on $\RR^k$ as translation. 
Particularly, $z$ acts trivially on the boundary $CM^{\infty} = \RR^k \times CM^{\infty}_1$.

Maximal flats of $M$ 
are the intersection of maximal flats of $P_n$ with $G$, and 
are given by $e^{i\mathfrak{t}}$ for
maximal commutative subspaces (maximal tori) $\mathfrak{t}$ of $\mathfrak{u}$. 
Fix a maximal torus $\mathfrak{t}$ of $\mathfrak{u}$.
Then, any maximal flat is written as $F(g) := \{ g e^\lambda g^{\dagger} \mid \lambda \in i \mathfrak{t}\}$ for $g \in G$,
where $\lambda \mapsto g e^\lambda g^{\dagger}$ is an isometry from 
Euclidean space $i \mathfrak{t}$ to $F(g)$.
The dimension $d$ of $i \mathfrak{t}$ is equal to the rank of $M$.
Via $\lambda \mapsto e^{\infty \lambda}$, 
we regard $i \mathfrak{t}$ as a subset, particularly, an apartment of 
$CM^{\infty}$.
Let $i \mathfrak{t}^+$ be any fixed (asymptotic) Weyl chamber in $i \mathfrak{t}$, 
and let $B$ denote the minimal parabolic subgroup (Borel subgroup) for $i \mathfrak{t}^+$.
Any point $p$ in $CM^{\infty}$ is written as 
$\lambda \cdot {\cal F}$ for $\lambda \in i \mathfrak{t}^+$ and 
${\cal F} \in G/B$.

Since $M = G \cap P_n$, 
any Weyl chamber of $M$ is written as the intersection of 
a Weyl chamber of $P_n$ and a maximal flat of $M$.
Consequently, $M^{\infty}$ is an isometric subspace of $P_n^{\infty}$.
By using the notation in Section~\ref{subsec:P_n}, 
for some $k_0 \in U(n)$,  
vectors $\lambda \in i \mathfrak{t}^+$ are written as $\bar \lambda \cdot [k_0]$, 
where $\bar \lambda$ ranges over a subspace of arranged vectors.

It is known \cite{BFGOWW,Woodward} that the Kempf-Ness function 
$f_v$ is convex on $M$.
Indeed,
consider the expression of $f_v$ in the maximal flat $F(g)$. 
Since $\{ e^{\lambda}\}_{\lambda \in i \mathfrak{t}}$ is a commutative subgroup, 
there is a finite set $\Omega(\pi)$ of vectors, called {\em weights},  in $i\mathfrak{t}$
such that matrices $\pi(e^{\lambda})= e^{\Pi(\lambda)}$ $(\lambda \in  i \mathfrak{t})$ are simultaneously 
diagonalized to
a diagonal matrix of diagonals $e^{\trace \omega \lambda}$ for $\omega \in \Omega(\pi)$.
Therefore, we have
\begin{equation}\label{eqn:f_v(ge^lambdag^dagger)}
f_v(g e^{\lambda} g^{\dagger}) = \log \sum_{\omega \in \Omega(\pi)} \|(\pi(g^{\dagger})v)_{\omega}\|^2 e^{\trace \omega \lambda } \quad (\lambda \in i \mathfrak{t}),  
\end{equation}
where $(\pi(g^{\dagger})v)_{\omega}$ denotes the orthogonal projection of $\pi(g^{\dagger})v$ to the eigenspace of $\omega$.
By Lemma~\ref{lem:log_sum_exp} and the expression (\ref{eqn:f_v(ge^lambdag^dagger)}), we have:
\begin{Lem} $f_v$ is convex, where:
	\begin{itemize}
	\item[(1)] The recession function $f_v^{\infty}$ is given by
	\begin{equation}\label{eqn:f_v^infty(p)}
	f_v^{\infty}(p) = \max \{ \trace \omega \lambda  \mid \omega \in \Omega(\pi) :(\pi(g^{\dagger})v)_{\omega}  \neq 0\}, 
	\end{equation}
	where $p = g e^{\lambda \infty} g^{\dagger}$ for $g \in G$ and $\lambda \in i \mathfrak{t}$.
	\item[(2)] $B_{CF(g)^{\infty}}(f_v^\infty)$ is the convex hull of $\omega$ over all $\omega \in \Omega(\pi)$ with $(\pi(g^{\dagger})v)_{\omega} \neq 0$, where $\omega$ are viewed as points in $CF(g)^{\infty}$ by $\omega \mapsto ge^{\omega \infty}g^{\dagger}$.  
\end{itemize}
\end{Lem} 
To study the boundedness of $f_v$, the following criterion is fundamental:
\begin{Thm}[{Hilbert-Mumford criterion; see \cite[Section 3.4.2]{Wallach_GIT}}]
		If $\inf_{g \in G} \|\pi(g)v\| = 0$, then there is $u \in i\mathfrak{u}$ 
		such that $\lim_{t \to \infty} \|\pi(e^{tu})v\| = 0$.
\end{Thm}
The reference~\cite[Theorem 3.23]{Wallach_GIT} also includes an elementary proof.
As noticed in \cite{KLM2009JDG,Woodward}, 
the nonnegativity of the asymptotic slope function of $f_v$ is equivalent to
the Hilbert-Mumford criterion:
\begin{Thm}[see \cite{KLM2009JDG,Woodward}]\label{thm:main0} The  following conditions are equivalent:
	\begin{itemize}
		\item[(a)] $\inf_{x \in M} \|\nabla f_v(x)\|_x = 0$.
		\item[(b)] $\inf_{x \in M} f_v(x) > -\infty$.
		\item[(c)] $0 \in B(f^{\infty}_v)$.
	\end{itemize}
\end{Thm}
The equivalence (a) $\Leftrightarrow$ (b) is known as the Kempf-Ness theorem~\cite{KempfNess}, 
and is called the {\em noncommutative duality} in \cite{BFGOWW}.
\begin{proof}
	We have already seen (a) $\Leftrightarrow$ (c) and (b) $\Rightarrow$ (c) in general situation; see Lemma~\ref{lem:unbounded} and Theorem~\ref{thm:char_of_B(f^infty)}.
	We verify (c) $\Rightarrow$ (b). 
	Suppose that $\inf_{x \in M} f_v(x) = -\infty$. By the Hilbert-Mumford criterion, there is $u \in i\mathfrak{u}$ 
	such that $\lim_{t \to \infty} f_v(e^{tu}) = - \infty$.
	Consider a maximal flat $F$ containing geodesic $t \mapsto e^{tu}$.
	Then $f_v$ is unbounded on $F$. By Lemma~\ref{lem:log_sum_exp}, we have $0 \not \in  B((f_v)_{F}^{\infty}) =B_{CF^{\infty}}(f_v^{\infty})$.
	By Proposition~\ref{prop:C_cap_B(h)}, we have $0 \not \in B(f_v^{\infty})$.
\end{proof}


In particular, a one-parameter subgroup $t \to e^{tu}$ in the Hilbert-Mumford criterion 
can be found by 
convex optimization of $f_v^{\infty}$ on Euclidean building $CM^{\infty}$:
\begin{equation}\label{eqn:opt_nullcone}
\mbox{inf.} \quad f_v^{\infty}(u) \quad {\rm s.t.} \quad u \in  U,\\
\end{equation}
where $U$ is any convex  neighborhood of the origin.
%

%
We are going to extend Theorem~\ref{thm:main0} for $f_v + b_p$
with giving a whole description of $B(f_v^{\infty})$ and $\dom f_v^*$.
For this,
we need a representation theoretic interpretation of Busemann functions.
By a {\em weight}
we mean a point in $i \mathfrak{t}$ that arises as a weight of some representation.
It is known that the set of 
weights is a discrete subgroup ({\em weight lattice}) in $i \mathfrak{t}$, 
and is generated by weights in $i \mathfrak{t}^+$. 
Any weight $\lambda$ in $i\mathfrak{t}^+$ determines 
an irreducible representation $\pi_{\lambda}$ of $G$ 
such that $\lambda$ is a highest weight.
The eigenspace for $\lambda$ is one dimensional, and
the unit eigenvector is denoted by $v_{\lambda}$.
\begin{Lem}\label{lem:b_-lambda_infty}
	For a weight $\lambda$, it holds
		$b_{e^{- \lambda \infty}}(g^{\dagger}g) = \log \|\pi_{\lambda}(g) v_\lambda\|^2 \quad (g \in G)$.
\end{Lem}
\begin{proof}
	Consider Iwasawa decomposition $g = k e^{\nu} n$ for $k \in K, \nu \in i \mathfrak{t}, n \in N$, where $N$ is interpreted as the horospherical subgroup 
	for $i \mathfrak{t}^+$ (see \cite[Section 2.17]{Eberlien}).
	From $\pi_{\lambda}(n)v_\lambda = v_\lambda$ (see \cite[Theorem 5.5]{KanppLieGroup}), 
	the RHS equals $2 \trace \nu \lambda$.
	On the other hand, $n^{\dagger}$ is an element of 
	the horospherical subgroup of the opposite Weyl chamber $- i\mathfrak{t}^+$,
	since $\lim_{t \to \infty} d(n^{-\dagger}e^{-\lambda t}n^{-1},e^{-\lambda t}) = \lim_{t \to \infty} d(ne^{\lambda t}n^{\dagger},e^{\lambda t}) = 0$ for $\lambda \in t \mathfrak{t}^+$. 
	Then the LHS equals $b_{e^{- \lambda \infty}}(n^{\dagger}e^{2\nu}n) = b_{e^{- \lambda \infty}}(e^{2\nu}) = - 2 \trace (- \lambda \nu)$ (by Example~\ref{ex:Euclidean_Busemann}).
\end{proof}

Therefore, the minimization of $f_v + b_p$ is essentially 
the $p$-{\em scaling problem} in~\cite{BFGOWW}.
A point $\nu$ in $i \mathfrak{t}$ is said to be {\em rational} if $\alpha \nu$ 
is a weight for some positive integer $\alpha$.
Let $s: CM^{\infty} \to i \mathfrak{t}^+$ denote the projection $\lambda \cdot {\cal F} \mapsto \lambda$.
\begin{Prop}\label{prop:f_v^infty}
		There is a finite set $\Lambda \subseteq i \mathfrak{t}^+$ (independent of $v$) such that
		\begin{equation}
		B(f^{\infty}_v) = \{ p \in CM^{\infty} \mid \langle \nu \cdot {\cal F}, p \rangle \leq f^{\infty}_{v}(\nu \cdot {\cal F})\ (\nu \in \Lambda, {\cal F} \in G/B)\}.
		\end{equation}
	For any Weyl chamber $C$, the projection 
		$s(C \cap B(f^{\infty}_v))$ 
          is a rational convex polytope.
\end{Prop}
\begin{proof}
	Take a Weyl chamber $C$ of $CP^{\infty}_{n}$,  
	and consider all apartments $E$ containing $C$.
	When all $E$ are regarded as $\RR^d$ with a common convex cone $C$,
	by Proposition~\ref{prop:C_cap_B(h)},  
	$C\cap B(f^{\infty}_v)$ is the intersection of $C$ and  
	finitely many (integral) polytopes $B_{E}(f_v^\infty)$, which are convex hulls of 
	finite subsets of weights in $\Omega(\pi) \subseteq \RR^n$. 
	Consequently, $s(C \cap B(f^{\infty}_v))$ 
	is a rational convex polytope.
	
	In particular, the affine span of a facet of $B_E(f_v^{\infty})$ is spanned by a subset of $\Omega(\pi)$, and its normal vector is chosen from $i\mathfrak{t}$. 
	The corresponding inequality is written as $\langle \nu \cdot {\cal F}, p \rangle 
	\leq f_v^{\infty}(\nu \cdot {\cal F})$ for 
	some ${\cal F} \in G/B$ and $\nu \in i \mathfrak{t}^+$ (with $\nu \cdot {\cal F} \in F$).
	Thus, $\Lambda$ can be chosen as a (finite) set of vectors arising as normal vectors of $d-1$-spaces spanned by subsets of $\Omega(\pi)$. 
\end{proof}

\begin{Prop}\label{prop:generalversion}
	$\dom f_v^* = B(f_v^{\infty})$.
\end{Prop}
\begin{proof}
	By rationality and convexity (Proposition~\ref{prop:f_v^infty}), 
	it suffices to show that for $p= \lambda\cdot {\cal F} \in CM^{\infty}$ with rational $\lambda \in i \mathfrak{t}^+$  
	it holds  $\inf_{x\in M}(f_v+ b_{p})(x) > - \infty$ if and only if 
	$p \in B(f_v^{\infty})$.
	Suppose that $\lambda = \nu/\alpha$ where $\nu$ is a weight and $\alpha$ is a positive integer.
	By Lemma~\ref{lem:b_-lambda_infty}, $\alpha b_p = b_{\nu \cdot {\cal F}}$ is 
	the Kempf-Ness function for some representation $\pi_{\nu'}$ and vector $v'$.
	Then $\alpha f_v + b_{\nu \cdot {\cal F}}$ is the Kempf-Ness function 
	for representation $\overbrace{\pi \otimes \pi \otimes \cdots \otimes \pi}^{\alpha} \otimes \pi_{\nu'}$
	and vector $\overbrace{v \otimes v \otimes \cdots \otimes v}^{\alpha} \otimes v'$; 
	see~\cite[Section 3.6]{BFGOWW}.
	Therefore, by Theorem~\ref{thm:main0}, we have $\inf_{x\in M}(f_v+ b_{p})(x) = \inf_{x\in M}(1/\alpha)(\alpha f_v+ b_{\nu \cdot {\cal F}})(x) > -\infty$ 
	$\Leftrightarrow$ $0 \in B((\alpha f_v+ b_{\nu \cdot {\cal F}})^{\infty})$ 
	$\Leftrightarrow$ $0 \in B(f_v^{\infty}+ b_{p}^{\infty})$ $\Leftrightarrow$ $p \in B(f_v^{\infty})$.
\end{proof}
Summarizing, we obtain a convex analysis formulation of $p$-scalability.
\begin{Thm}\label{thm:main} For $p \in CM^{\infty}$, the following conditions are equivalent:
	\begin{itemize}
		\item[(a)] $\inf_{x \in M} \|\nabla (f_v + b_p)(x)\|_x = 0$.
		\item[(b)] $- f^*(p) = \inf_{x \in M} (f_v +b_p)(x) > - \infty$.
		\item[(c)] $p \in B(f_v^{\infty})$. 
	\end{itemize}
\end{Thm}

We finally consider the moment polytope membership.
Here, $\overline{s \nabla^{\infty}f_v(M)}$ is nothing but the {\em moment polytope} for $\pi,v$ in the sense of \cite{BFGOWW}.\footnote{This fact can be seen from  Proposition~\ref{prop:nabla^inf_P_n} and the fact that 
the {\em moment map} $\mu: \pi(G)v  \setminus \{0\} \to i \mathfrak{u} (=T_I)$ 
in \cite{BFGOWW} is written as $\mu(\pi(g)v) = g^\dagger df_v(gg^\dagger) g$.
}
\begin{Lem}\label{lem:momentpolytope}
	$\overline{s \nabla^{\infty}f_v(M)} = \bigcup_{g \in G} i \mathfrak{t}^+ \cap B(f^{\infty}_{\pi(g)v})$.
\end{Lem}
\begin{proof}
    Proposition~\ref{prop:nabla^inf_P_n} holds in this setting 
	by replacing $\diag \lambda$ with $\lambda \in i \mathfrak{t}^+$ 
	and $bk$ 
	with Iwasawa decomposition $bk$ ($b \in B, k \in K$). 
	By (\ref{eqn:sB(f^infty)}), it suffices to show that $i \mathfrak{t}^+ \cap B(f_{\pi(g)v}^{\infty}) = s(C \cap B(f^{\infty}_v))$ for $g \in G$ and Weyl chamber $C = g^{\dagger} (i\mathfrak{t}^+)$.
	Indeed, from $f_{\pi(g)v}^{\infty}(u) = f_{v}^{\infty}(g^{\dagger} u)$ (by (\ref{eqn:f_v^infty(p)})), 
	we have $\lambda \in i\mathfrak{t}^+ \cap B(f_{\pi(g)v}^{\infty})$ 
	$\Leftrightarrow$ $\langle u, \lambda \rangle \leq f_{\pi(g)v}^{\infty}(u)$ $(\forall u \in M^{\infty})$
	$\Leftrightarrow$  $\langle g^{\dagger} u, g^{\dagger }p\rangle \leq f_{v}^{\infty}(g^{\dagger}u)$ $(\forall u \in M^{\infty})$ $\Leftrightarrow$ $\langle u', g^{\dagger}\lambda\rangle \leq f_{v}^{\infty}(u')$ $(\forall u' \in M^{\infty})$ $\Leftrightarrow$ $g^{\dagger}\lambda \in B(f_v^{\infty})$  $\Leftrightarrow$ $\lambda \in s(C \cap B(f_v^{\infty}))$, 
	where $g^{\dagger}\lambda$ is written as $\lambda \cdot {\cal F}$ for ${\cal F} \in G/B$ corresponding to $C$. 
\end{proof}

The polytope $ i \mathfrak{t}^+ \cap B(f^{\infty}_{\pi(g)v}) = s(C \cap B(f^{\infty}_v))$  is 
what should be called the {\em Borel polytope}; see \cite{BFGOWW_tensor}.
Notice that the approximate $(\lambda,\mu)$-scalability of the previous section 
is nothing but the moment polytope membership $(\lambda, \mu) \in  
\overline{s\nabla^{\infty}f_A(P_n \times P_n)}$.

The convexity theorem of the moment polytope says:
\begin{Thm}[Convexity theorem~\cite{GuilleminSternberg}]\label{thm:convexity_thm}
	The moment polytope $\overline{s \nabla^{\infty}f_v(M)}$ is a rational convex polytope. 
\end{Thm}
The {\em shifting trick}~\cite{Brion,NessMumford} 
reduces the membership of the moment polytope 
to a single optimization problem.

\begin{Thm}[{Shifting trick~\cite{Brion,NessMumford}}]\label{thm:shifting}
	A rational vector $\lambda \in i \mathfrak{t}^+$ belongs to $\overline{s \nabla^{\infty}f_v(M)}$ if and only if $\inf_{x \in M} (f_{\pi(g) v} + b_{e^{\lambda \infty}})(x) > -\infty$ for generic $g \in G$.
\end{Thm}

We prove a slightly stronger statement from our formulation, which implies Theorems~\ref{thm:convexity_thm} and \ref{thm:shifting}.

\begin{Thm}
	$\overline{s \nabla^{\infty}f_v(M)} = i \mathfrak{t}^+ \cap B(f^{\infty}_{\pi(g)v})$ for generic $g \in G$.
\end{Thm}
Our proof is a direct adaptation of \cite[Lemma 54 and Proposition 55]{Franks18}.
\begin{proof}
Let  $\Lambda \subseteq i \mathfrak{t}^+$ be a finite set in Proposition~\ref{prop:f_v^infty}.  
	Let $g \in G$. 
	Then $i \mathfrak{t}^+ \cap B(f^{\infty}_{\pi(g)v})$ is 
	the set of $\lambda  \in i \mathfrak{t}^+$ satisfying
	\begin{equation}\label{eqn:inequality1}
	\langle \nu \cdot {\cal F}, \lambda \rangle \leq f_{\pi(g)v}^{\infty}(\nu \cdot {\cal F}) \quad (\nu \in \Lambda, {\cal F} \in G/B).
	\end{equation}
	We use the notation $\lambda  = \bar \lambda \cdot [k_0]$ 
	to deduce an explicit inequality description.
	Represent $\nu \cdot {\cal F}$ 
	as $\nu \cdot {\cal F} = \bar \nu \cdot [g^{-\dagger}h^{\dagger}k_0]$ for $h \in G$. 
	Then $\langle \nu \cdot {\cal F}, \lambda \rangle = \langle \bar \nu \cdot [g^{-\dagger}h^{\dagger}k_0],  \bar \lambda\cdot [k_0]  \rangle = \langle \bar \nu \cdot [k_0^{\dagger} g^{-\dagger}h^{\dagger}k_0],  \bar \lambda\cdot [I]  \rangle$, and 
	$f_{\pi(g)v}^{\infty}(\nu \cdot {\cal F}) = f_{\pi(g)v}^{\infty}(\bar \nu \cdot [g^{-\dagger}h^{\dagger}k_0]) = f_{\pi(h)v}^{\infty}(\nu)$, where $[I] = \{E_i\}$ is the standard flag.
		Let ${\cal U}^{g,h} = \{U_j^{g,h}\}$ be the flag generated by $ k_0^{\dagger}g^{-\dagger}h^{\dagger}k_0$. 
		By Lemma~\ref{lem:<>},  (\ref{eqn:inequality1}) is written as 
	\begin{equation}\label{eqn:inequality2}
\sum_{i,j=1}^{n} (\bar \lambda_i - \bar \lambda_{i+1}) (\bar \nu_j- \bar \nu_{j+1}) \dim E_i \cap U_j^{g,h}  \leq f_{\pi(h)v}^{\infty}(\nu) \quad (\nu \in \Lambda, h \in G).
	\end{equation}
	
	We next consider the quantities $f_{\pi(h)v}^{\infty}(\nu)$ and $\dim E_i \cap U_j^{g,h}$ 
	involving $h$. 
	By (\ref{eqn:f_v^infty(p)}), 
	the former quantity
	$f_{\pi(h)v}^{\infty}(\nu)$ 
    takes a value from 
	finite set $A_{\nu}  := \{ \trace \nu \omega \mid \omega \in \Omega(\pi) \}$.
	For $\alpha \in A_{\nu}$, let $G_{\nu,\alpha} \subseteq G$ be the affine subvariety 
	of  consisting of $h$ with $f_{\pi(h)v}^{\infty}(\nu) \leq \alpha$, which is 
	defined by algebraic conditions $(\pi(h)v)_\omega = 0$ for all $\omega \in \Omega(\pi)$ with $\trace \nu \omega > \alpha$.
    The latter quantity $\dim E_i \cap U_j^{g,h}$ takes a value in $\{0,1,2,\ldots,n\}$. 
	Let $D$ denote the set of all $n \times n$ matrices $d = (d_{ij})$ such that each entry $d_{ij}$ is one of $0,1,2,\ldots,n$, where
	the partial order $\leq$ on $D$ is defined by 
	$d \leq d'$ $\Leftrightarrow $ $d_{ij} \leq d_{ij}'$ $(\forall i,j)$.
	For $\nu \in \Lambda, \alpha \in A_{\nu}$, 
	let $D_{\nu,\alpha}(g) \subseteq D$ be the set of all $d=(d_{ij})$ such that
	there is $h \in G_{\nu,\alpha}$ such that $d_{ij} = \dim E_i \cap U_j^{g,h}$ for $1\leq i,j \leq n$, where $d_{nj} = d_{jn} = j$ holds for all $j \in [n]$.
	Let $\bar{D}_{\nu,\alpha}(g) \subseteq D_{\nu,\alpha}(g)$ denote the set of maximal members with respect to $\leq$.
	Then (\ref{eqn:inequality2}) is written as
	\begin{equation}\label{eqn:inequality3}
	\sum_{i,j=1}^n (\bar \lambda_i - \bar \lambda_{i+1}) (\bar \nu_j- \bar \nu_{j+1}) d_{ij}  \leq \alpha \quad (\nu \in \Lambda, \alpha \in A_{\nu}, d\in \bar{D}_{\nu,\alpha}(g)).
	\end{equation}
	
	Let $S_{\nu,\alpha,d}$ be the subvariety of $G \times G_{\nu,\alpha}$ 
	consisting of $g,h$ with $\dim E_{i} \cap U_j^{g,h} \geq d_{ij}$ for $i,j$, 
	which is defined by vanishing of subdeterminants of $k_0^{\dagger}g^{-\dagger}h^{\dagger}k_0$. Indeed,
	$\dim E_{i} \cap U_j^{g,h}$ is $j$ minus the rank of 
	lower left $(n-i) \times j$ submatrix
	$k_0^{\dagger}g^{-\dagger}h^{\dagger}k_0$. 
	Let $\pi$ be the projection $(g,h) \mapsto g$. We claim:
	\begin{itemize}
		\item[($*$)] $\pi (S_{\nu,\alpha,d})$ is an affine subvariety of $G$.
    \end{itemize}
    The proof is given in the end. Let $D^*_{\nu,\alpha} \subseteq D$ be the set of all maximal $d$ with 
    $\pi (S_{\nu,\alpha,d}) = G$.  Consider the set $Q^*$ of $\lambda \in i \mathfrak{t}^+$ satisfying
    \begin{equation}\label{eqn:inequality4}
    \sum_{i,j=1}^n (\bar \lambda_i - \bar \lambda_{i+1}) (\bar \nu_j- \bar \nu_{j+1}) d_{ij}  \leq \alpha \quad (\nu \in \Lambda, \alpha \in A_{\nu}, d\in D^*_{\nu,\alpha}).
    \end{equation}
    For $d\in D^*_{\nu,\alpha}$, there is $d' \in D_{\nu,\alpha}(g)$ with $d \leq d'$.
    That is, (\ref{eqn:inequality4}) is looser than (\ref{eqn:inequality3}). 
    Thus, $Q^*$ contains $i \mathfrak{t}^+ \cap B(f^{\infty}_{\pi(g)v})$ for every $g \in G$.
    Consider the finite union $H := \bigcup_{\nu \in \Lambda, \alpha \in A_{\nu}, d \in D \setminus D_{\nu,\alpha}^*} \pi(S_{\nu,\alpha,d})$, which is a proper subvariety of $G$.  
    Then we can choose a {\em generic} $g^* \in G \setminus H$. 
   For such $g^*$, it must hold $\bar D_{\nu,\alpha}(g^*) = D^*_{\nu,\alpha}$ for 
   all $\nu \in \Lambda, \alpha \in A_{\nu}$.
   This means $Q^* = i \mathfrak{t}^+ \cap B(f^{\infty}_{\pi(g^*)v}) = \bigcup_{g \in G} i \mathfrak{t}^+ \cap B(f^{\infty}_{\pi(g)v}) = \overline{s\nabla^{\infty}f_v(M)}$.

    Finally we verify ($*$). By the closure theorem (see \cite[Section 4.7, Theorem 7]{CoxLittleOShea}), 
    $\pi(G_{\nu,\alpha,d})$ is a constructible set, i.e., 
    it is an affine variety $Z_0$ minus an affine variety $Z'$.
    We show that $\pi(G_{\nu,\alpha,d})$ 
    is a closed set in the Euclidean topology, which implies
    that $\pi(G_{\nu,\alpha,d}) = Z_0$ is an affine variety.
    Consider a sequence $g_1,g_2,\ldots$ in  $\pi(G_{\nu,\alpha,d})$ 
    converging to $g \in G$.
    For each $i$, there is $h_k \in G$ with
    such that $f_{\pi(h_k)v}(\nu) \leq \alpha$ and
    $\dim E_i \cap U_j^{g,h_k} \geq d_{ij}$ $(i,j \in [n])$.
    These quantities are determined by $\bar \nu \cdot [g^{-\dagger} h_k^{\dagger} k_0] = g^{-\dagger} h^{\dagger}_k  \nu$.
    By Iwasawa decomposition, 
    $h_k$ can be chosen from the compact group $K$.
    By taking a subsequence, we may assume that $h_k$ converges to $h \in K$.
    From $h_k \in G_{\nu,\alpha}$ for each $k$,
    it is clear that $h \in G_{\nu,\alpha}$.
    As mentioned, the condition $\dim E_i \cap U_j^{g_k,h_k} \geq d_{ij}$ is 
    written as vanishing of subdeterminants of 
    $k_0^{\dagger}g^{-\dagger}_k h^{\dagger}_kk_0$. 
    These subdeterminants vanish in the limit $k_0^{\dagger}g^{-\dagger}h^{\dagger} k_0$ as well. 
    Then $\dim E_i \cap U_j^{g,h} \geq d_{ij}$.
    Thus $(g,h) \in S_{\nu,\alpha,d}$, and  $\pi(S_{\nu,\alpha,d})$ is closed.
\end{proof}
In particular, $B(f_v^{\infty})$ contains the moment polytope in a ``generic" chamber.
Thus, the moment polytope membership 
for a given vector $\lambda \in i \mathfrak{t}^+$  
also reduces, after taking generic $g \in G$, 
to the convex optimization problem on Euclidean building $CM^{\infty}$: 
\begin{equation}\label{eqn:opt_moment}
\mbox{inf.} \quad f_{\pi(g)v}^{\infty}(u) - \langle p, u \rangle \quad {\rm s.t.} \quad u \in  U,
\end{equation}
where $p := e^{\infty \lambda}  \in CM^{\infty}$ and $U$ is any convex neighborhood of $0$.

This gives rise to a challenging research problem to develop algorithms 
solving convex optimization problems (\ref{eqn:opt_nullcone}), (\ref{eqn:opt_moment}).
On a single Weyl chamber $C$ (or an apartment), 
it is a usual Euclidean convex optimization. 
However, at a boundary point $q$ of $C$, 
one have to search a descent direction from infinitely many Weyl chambers containing $q$. 
This seems impossible in principle. 
So one have to exploit and utilize special properties 
of the objective function, particularly, the recession function of the Kempf-Ness function, as in \cite{HamadaHirai}. Moreover, to keep variable 
$q = \lambda \cdot {\cal F}$, 
one should keep basis vectors of flag ${\cal F}$ with bounded bit-length.
This is also a highly nontrivial problem. 
A recent work~\cite{FranksSomaGoemans2022} 
for finding a violating vector subspace in (\ref{eqn:Franks}) 
may give hints toward this direction.

\section*{Acknowledgments}
The author thanks Hiroyuki Ochiai for helpful discussion, and thanks for Zhiyuan Zhan for corrections. The author also thanks Harold Nieuwboer and Michael Walter for discussion on the Legendre-Fenchel duality.
The work was partially supported by JST PRESTO Grant Number JPMJPR192A, Japan.

\end{document}